\documentclass[11pt,a4paper,leqno]{amsart}

\title
[Propagation of anisotropic Gabor singularities]
{Propagation of anisotropic Gabor singularities for Schr\"odinger type equations}

\author[M. Cappiello]{Marco Cappiello}

\address{Dipartimento di Matematica, Universit\`a di Torino, Via Carlo Alberto 10,
10123 Torino, Italy}

\email{marco.cappiello[AT]unito.it}

\author[L. Rodino]{Luigi Rodino}

\address{Dipartimento di Matematica, Universit\`a di Torino, Via Carlo Alberto 10,
10123 Torino, Italy}

\email{luigi.rodino[AT]unito.it}

\author[P. Wahlberg]{Patrik Wahlberg}

\address{Dipartimento di Scienze Matematiche, Politecnico di Torino, Corso Duca degli Abruzzi 24,
10129 Torino, Italy}

\email{patrik.wahlberg[AT]polito.it}

\usepackage{latexsym}
\usepackage[a4paper]{geometry}
\usepackage{amsmath}
\usepackage{amssymb}
\usepackage{amsthm}
\usepackage{bbm}
\usepackage{amsfonts}
\usepackage{mathrsfs}
\usepackage{calc}
\usepackage{cite}
\usepackage{color}
\usepackage{graphicx}
\usepackage{enumerate}

\numberwithin{equation}{section}          

\newtheorem{thm}{Theorem}
\numberwithin{thm}{section}

\newcommand{\rubrik}{}
\newtheorem{prop}[thm]{Proposition}
\newtheorem{cor}[thm]{Corollary}
\newtheorem{lem}[thm]{Lemma}

\theoremstyle{definition}

\newtheorem{defn}[thm]{Definition}
\newtheorem{example}[thm]{Example}

\theoremstyle{remark}

\newtheorem{rem}[thm]{Remark}              

\newcommand{\scal}[2]{\langle #1,#2\rangle}

\newcommand{\pdd}[2] {\partial_{#1} ^{#2}}

\newcommand{\ro}{\mathbf R}
\newcommand{\no}{\mathbf N}

\newcommand{\rr}[1]{\mathbf R^{#1}}
\newcommand{\sr}[1]{\mathbf S^{#1}}
\newcommand{\sro}[1]{\mathbf S}
\newcommand{\nn}[1]{\mathbf N^{#1}}

\newcommand{\dd}{\mathrm {d}}

\newcommand{\nm}[2]{\Vert #1\Vert _{#2}}

\newcommand{\ep}{\varepsilon}
\newcommand{\fy}{\varphi}

\newcommand{\cdo}{\, \cdot \, }

\newcommand{\supp}{\operatorname{supp}}

\newcommand{\wpr}{{\text{\footnotesize $\#$}}}

\newcommand{\eabs}[1]{\langle #1\rangle}

\newcommand{\GL}{\operatorname{GL}}

\newcommand{\charac}{\operatorname{char}}

\newcommand{\rB}{\operatorname{B}}

\newcommand{\WF}{\mathrm{WF}}
\newcommand{\WFg}{\mathrm{WF_{\rm g}}}
\newcommand{\WFgs}{\mathrm{WF_{g}^{\it \sigma}}}

\newcommand{\cS}{\mathscr{S}}

\newcommand{\cT}{\mathcal{T}}

\newcommand{\cD}{\mathscr{D}}
\newcommand{\cF}{\mathscr{F}}

\newcommand{\cK}{\mathscr{K}}

\newcommand{\J}{\mathcal{J}}
\newcommand{\wt}{\widetilde}
\newcommand{\wh}{\widehat}
\newcommand{\re}{{\rm Re}}
\newcommand{\im}{{\rm Im}}
\def\la{\langle}
\def\ra{\rangle}
\newcommand{\leqs}{\leqslant}
\newcommand{\geqs}{\geqslant}

\begin{document}

\begin{abstract}
We show results on propagation of anisotropic Gabor wave front sets
for solutions to a class of evolution equations of Schr\"odinger type. 
The Hamiltonian is assumed to have a real-valued principal symbol
with the anisotropic homogeneity $a(\lambda x, \lambda^\sigma \xi) = \lambda^{1+\sigma} a(x,\xi)$
for $\lambda > 0$ where $\sigma > 0$ is a rational anisotropy parameter.
We prove that the propagator is continuous on anisotropic Shubin--Sobolev spaces. 
The main result says that the propagation of the anisotropic Gabor wave front set follows the Hamilton flow 
of the principal symbol.
\end{abstract}


\maketitle

\section{Introduction}\label{sec:intro}

We prove results on propagation of anisotropic phase space singularities for the initial value Cauchy problem for evolution equations of the form 
\begin{equation}\label{eq:schrodeq0}
\left\{
\begin{array}{rl}
\partial_t u(t,x) + i a^w(x,D_x) u (t,x) & = 0, \quad x \in \rr d, \quad t \in [-T,T] \setminus \{ 0\} , \\
u(0,\cdot) & = u_0.
\end{array}
\right.
\end{equation}
Here $T > 0$, $a^w(x,D_x)$ is a Weyl pseudodifferential operator and $u_0 \in \cS'(\rr d)$ is a tempered distribution.  

The Hamiltonian $a^w(x,D_x)$ is assumed to have real-valued principal symbol $a_0$. 
Following the fundamental idea of H\"ormander we show that the singularities at time $t \in [-T,T]$
are the singularities of the initial datum $u_0$ transported by the Hamilton flow $\chi_t $ of the principal symbol $a_0$. 
The Hamilton flow $( x(t), \xi (t) ) = \chi_t(x,\xi)$ is the solution to Hamilton's equation with initial datum $(x,\xi) \in T^* \rr d \setminus \{ (0,0) \} $, that is the solution to the system of ordinary differential equations 
\begin{equation*}
\left\{
\begin{array}{l}
x' (t) = \nabla_{\xi} a_0 \left( x(t), \xi(t) \right), \\
\xi' (t) = -\nabla_{x} a_0 \left( x(t), \xi(t) \right), \\
x(0) = x, \\
\xi(0) = \xi. 
\end{array}
\right.
\end{equation*}

The concept of phase space singularities that we use is the anisotropic Gabor wave front set, 
which is determined by an anisotropy parameter $\sigma > 0$. 
For $u \in \cS'(\rr d)$ the anisotropic Gabor wave front set $\WFgs (u)$ is a $\sigma$-conic closed subset of $T^* \rr d \setminus 0$. 
A $\sigma$-conic subset of $T^* \rr d \setminus 0$ contains anisotropic phase space curves of the form
\begin{equation}\label{eq:anisotropiccurve}
\lambda \mapsto ( \lambda x, \lambda^\sigma \xi ) \in T^* \rr d \setminus 0, \quad \lambda > 0, 
\end{equation}
if one point of the curve belongs to the subset. 

The anisotropic Gabor wave front set $\WFgs (u)$ is defined by means of the short-time Fourier transform 
$V_\fy u (x, \xi) = \cF \left(  u \, \overline{\fy(\cdot-x) }\right) (\xi)$ where $\fy \in \cS(\rr d) \setminus \{ 0 \}$ 
is a window function. 
To wit $z_0 = (x_0,\xi_0) \in T^* \rr d \setminus 0$ satisfies $z_0 \notin \WFgs ( u )$ 
if there exists an open set $U \subseteq T^* \rr d$ such that $z_0 \in U$ and 
\begin{equation*}
\sup_{(x,\xi) \in U, \ \lambda > 0} \lambda^N |V_\fy u (\lambda x, \lambda^\sigma \xi)| < + \infty \quad \forall N \geqs 0. 
\end{equation*}

This means that the short-time Fourier transform, which a priori is polynomially upper bounded, 
decays superpolynomially along curves of the form \eqref{eq:anisotropiccurve} in a neighborhood of $z_0$. 
For $u \in \cS'(\rr d)$ we have $\WFgs (u) = \emptyset$ if and only if $u \in \cS(\rr d)$ so $\WFgs (u)$ measures globally singular behavior
in the sense of lack of smoothness or decay at infinity comprehensively. 

We impose the condition that the Hamiltonian $a^w(x,D)$ has a real-valued principal symbol $a_0$ which satisfies the anisotropic homogeneity 
\begin{equation}\label{eq:anisotropiccondition}
a_0 (\lambda x, \lambda^\sigma \xi) = \lambda^{1 + \sigma} a_0(x,\xi), \quad (x,\xi) \in T^* \rr d \setminus 0, \quad \lambda > 0. 
\end{equation}
This condition turns out to have several beneficial consequences for the problem we study.

First it implies that the Hamilton flow $\chi_t$ of $a_0$ commutes with the anisotropic scaling map
\begin{equation*}
T^* \rr d \setminus 0 \ni (x,\xi) \mapsto ( \lambda x, \lambda^\sigma \xi ) \in T^* \rr d \setminus 0
\end{equation*}
for each $\lambda > 0$. This is a natural requirement for propagation results of the form $\WFgs (\cK_t u_0) \subseteq \chi_t \WFgs (u_0)$,
where $\cK_t u_0 = e^{- i t a^w(x,D)} u_0$ denotes the solution operator (propagator) for \eqref{eq:schrodeq0}, that we aim for, 
since $\WFgs (u)$ is $\sigma$-conic for all $u \in \cS'(\rr d)$. 

Secondly if $\sigma > 0$ is rational then condition \eqref{eq:anisotropiccondition} on the principal symbol allows us to prove the main result of this paper, that is the
propagation of singularities
\begin{equation}\label{eq:propsing}
\WFgs (\cK_t u_0) = \chi_t ( \WFgs (u_0) ), \quad t \in [-T,T], \quad u_0 \in \cS'(\rr d), 
\end{equation}
where $T > 0$. 

The term ``principal symbol'' refers here to the pseudodifferential calculus of anisotropic Shubin symbols \cite{Chatzakou1,Rodino4,Wahlberg4}. 
The symbols exhibit anisotropic behavior on phase space according to the assumed estimates
\begin{equation*}
|\pdd x \alpha \pdd \xi \beta a(x,\xi)|
\lesssim ( 1 + |x| + |\xi|^{\frac1\sigma} )^{m - |\alpha| - \sigma |\beta|}, \quad (x,\xi) \in T^* \rr d, \quad \alpha, \beta \in \nn d, 
\end{equation*}
where again $\sigma > 0$ is a given anisotropy parameter, and $m \in \ro$ is the order. These symbol classes are denoted $G^{m,\sigma}$. 

In the main result Theorem \ref{thm:propagationWFs} we show \eqref{eq:propsing} under the following assumptions. 
Suppose $k,m \in \no \setminus 0$, $\sigma = \frac{k}{m}$, 
and let $a \in G^{1 + \sigma,\sigma}$, 
$a \sim \sum_{j = 0}^{\infty} a_j$, 
where 
$a_0 \in C^\infty(\rr {2d} \setminus 0)$ is real-valued and satisfies \eqref{eq:anisotropiccondition}, whereas the lower order terms satisfy
$a_j \in G^{(1+\sigma) (1-2 j ), \sigma}$ for $j \geqs 1$. 
An example of a symbol that satisfies the criteria is
\begin{equation*}
a(x,\xi) = c \psi(x,\xi) \left( |x|^{2k} + |\xi|^{2m} \right)^{\frac12 \left( \frac1k + \frac1m \right)} 
\end{equation*}
where $c \in \ro \setminus 0$ and $\psi$ is a smooth function vanishing in a small ball around the origin in $T^* \rr d$. 

We show that the solution operator is continuous 
on anisotropic Shubin--Sobolev spaces.
This is of independent interest but also a tool for 
the proof of Theorem \ref{thm:propagationWFs}. 
The proof of the main result is based on ideas from \cite{Hormander1}. 
More precisely our result is an anisotropic version of \cite[Theorem~4.2]{Nicola2}
which treats propagation of the (isotropic) Gabor wave front set when the principal symbol is real-valued
and homogeneous of order two on $T^* \rr d$. With $\sigma = 1$ our result implies a weaker form of \cite[Theorem~4.2]{Nicola2}. 

The proof ideas for Theorem \ref{thm:propagationWFs} and \cite[Theorem~4.2]{Nicola2}
are based on H\"ormander's proof of \cite[Theorem~23.1.4]{Hormander1}. 
This result concerns Hamiltonians with first order H\"ormander type symbols, 
the continuity concerns classical Sobolev spaces,
and the singularities are the classical smooth wave front set. 
The proof techniques rely on energy estimates, functional analysis and pseudodifferential calculus. 
Our proofs in this paper are worked out in detail as opposed to the rather brief arguments in \cite[Chapter~23.1]{Hormander1} and \cite{Nicola2}.

We also prove the propagation \eqref{eq:propsing} for a different type of Hamiltonian of the form $a^w(x,D) = p(D) + \la v,x \ra$
where $p \in C^\infty(\rr d)$ is a sum of polynomials of each variable in $\rr d$, with real coefficients, 
of order $m \geqs 2$,
$v \in \rr d$ is a vector each of whose coordinate is nonzero, and $\sigma = \frac1{m-1}$. 
Since this setup includes the Airy operator $\frac{\dd^2}{\dd x^2} - x$ when $d =1$ we say that the corresponding equation \eqref{eq:schrodeq0}
is of Airy--Schr\"odinger type. 
Using results from \cite{Wahlberg3} we also formulate a version of \eqref{eq:propsing} in the Gelfand--Shilov space functional framework and corresponding anisotropic wave front sets \cite{Rodino3}. 

Denoting by $P_m$ the principal part of $p$, we show \eqref{eq:propsing}
where $\chi_t$ is the Hamilton flow of $P_m(\xi)$. This generalizes a particular case of \cite[Theorem~5.1]{Wahlberg4} where $v = 0$. 
Since $P_m(\xi)$ does not depend on $x$, 
the Hamiltonian flow for $P_m$ is trivial in the sense that it is constant in time with respect to the dual coordinates as
$\chi_t (x, \xi) = (x + t \nabla P_{m} (\xi), \xi)$.  
This contrasts to the
Hamilton flow in the main result Theorem \ref{thm:propagationWFs} where 
both space and dual coordinates may depend on time. 
The techniques we use for Airy--Schr\"odinger equations are an explicit formula for the Schwartz kernel of the propagator and general results on propagation of singularities from \cite{Rodino3,Rodino4,Wahlberg3,Wahlberg4}. 

Our results in this paper fit in a project to investigate globally anisotropic pseudodifferential operators \cite{Cappiello2,Cappiello4,Chatzakou1,Martin1,Rodino3,Rodino4}
and propagation of global singularities for evolution equations
\cite{PravdaStarov1,Wahlberg3,Wahlberg4}. 
The techniques are inspired from those of pseudodifferential operators defined by symbols that are anisotropic in the dual variables
for fixed space coordinates. 
These ideas have been investigated e.g. in \cite{Lascar1,Parenti1,Garello1}. 

A major new feature of our main result Theorem \ref{thm:propagationWFs} as opposed to 
earlier propagation results \cite{Wahlberg4}, is that it admits Hamiltonians that give 
rise to flows that are non-trivial in the sense that the dynamics involve all phase space coordinates. 

Concerning the organization of the paper, 
Section \ref{sec:prelim} contains notations, background concepts and conventions, 
and Section \ref{sec:AnisoShubinCalculus} recalls material on anisotropic Shubin pseudodifferential calculus. 
Section \ref{sec:ShubinSobolev} is devoted to Shubin--Sobolev modulation spaces in the anisotropic context, 
a recollection of localization operators, and an inequality of sharp G\aa rding type which is essential. 
In Section \ref{sec:airyschrodinger} we deduce propagation results for Airy--Schr\"odinger equations. 
Section \ref{sec:hamiltonflow} treats Hamiltonians that are anisotropically homogeneous as in \eqref{eq:anisotropiccondition}
and their Hamilton flows, and in Section \ref{sec:solutionsaniso} we show existence and uniqueness of solutions to an inhomogeneous form of \eqref{eq:schrodeq0}
in anisotropic Shubin--Sobolev spaces for Hamiltonian symbols in $G^{1+\sigma,\sigma}$ with bounded imaginary part and $\sigma > 0$ rational. 
Then Section \ref{sec:propagation} is dedicated to the main result on propagation of singularities, and finally Section \ref{sec:examples} consists of a very short discussion of examples.

\section{Preliminaries}\label{sec:prelim}

The unit sphere in $\rr d$ is denoted $\sr {d-1} \subseteq \rr d$. 
An open ball of radius $r > 0$ centered in $x \in \rr d$ is denoted $\rB_r (x)$, and $\rB_r(0) = \rB_r$.  
The transpose of a matrix $A \in \rr {d \times d}$ is denoted $A^T$ and the inverse transpose 
of $A \in \GL(d,\ro)$ is $A^{-T}$. 
We write $f (x) \lesssim g (x)$ provided there exists $C>0$ such that $f (x) \leqs C \, g(x)$ for all $x$ in the domain of $f$ and of $g$. 
If $f (x) \lesssim g (x) \lesssim f(x)$ then we write $f \asymp g$. 
We use the partial derivative $D_j = - i \partial_j$, $1 \leqs j \leqs d$, acting on functions and distributions on $\rr d$, 
with extension to multi-indices. 
The bracket $\eabs{x} = (1 + |x|^2)^{\frac12}$ for $x \in \rr d$ satisfies
Peetre's inequality with optimal constant \cite[Lemma~2.1]{Rodino3}, that is
\begin{equation}\label{eq:Peetre}
\eabs{x+y}^s \leqs \left( \frac{2}{\sqrt{3}} \right)^{|s|} \eabs{x}^s\eabs{y}^{|s|}\qquad x,y \in \rr d, \quad s \in \ro. 
\end{equation}
We use the normalization of the Fourier transform
\begin{equation*}
 \cF f (\xi )= \widehat f(\xi ) = (2\pi )^{-\frac d2} \int _{\rr
{d}} f(x)e^{-i\scal  x\xi } \, \dd x, \qquad \xi \in \rr d, 
\end{equation*}
for $f\in \cS(\rr d)$ (the Schwartz space), where $\scal \cdo \cdo$ denotes the scalar product on $\rr d$. 
The conjugate linear action of a distribution $u$ on a test function $\phi$ is written $(u,\phi)$, consistent with the $L^2$ inner product $(\cdo ,\cdo ) = (\cdo ,\cdo )_{L^2}$ which is conjugate linear in the second argument. 

Denote translation by $T_x f(y) = f( y-x )$ and modulation by $M_\xi f(y) = e^{i \scal y \xi} f(y)$ 
for $x,y,\xi \in \rr d$ where $f$ is a function or distribution defined on $\rr d$. 
If $\fy \in \cS(\rr d) \setminus \{0\}$ then
the short-time Fourier transform (STFT) of a tempered distribution $u \in \cS'(\rr d)$ is defined by 
\begin{equation*}
V_\fy u (x,\xi) = (2\pi )^{-\frac d2} (u, M_\xi T_x \fy) = \cF (u T_x \overline \fy)(\xi), \quad x,\xi \in \rr d. 
\end{equation*}
The function $V_\fy u$ is smooth and polynomially bounded \cite[Theorem~11.2.3]{Grochenig1}, that is 
there exists $k \geqs 0$ such that 
\begin{equation}\label{eq:STFTtempered}
|V_\fy u (x,\xi)| \lesssim \eabs{(x,\xi)}^{k}, \quad (x,\xi) \in T^* \rr d.  
\end{equation}
We have $u \in \cS(\rr d)$ if and only if
\begin{equation}\label{eq:STFTschwartz}
|V_\fy u (x,\xi)| \lesssim \eabs{(x,\xi)}^{-N}, \quad (x,\xi) \in T^* \rr d, \quad \forall N \geqs 0.  
\end{equation}

The transform inverse to the STFT is given by
\begin{equation}\label{eq:STFTinverse}
u = (2\pi )^{-\frac d2} \iint_{\rr {2d}} V_\fy u (x,\xi) M_\xi T_x \fy \, \dd x \, \dd \xi
\end{equation}
provided $\| \fy \|_{L^2} = 1$, with action under the integral understood, that is 
\begin{equation}\label{eq:moyal}
(u, f) = (V_\fy u, V_\fy f)_{L^2(\rr {2d})} = (V_\fy^* V_\fy u, f)
\end{equation}
for $u \in \cS'(\rr d)$ and $f \in \cS(\rr d)$, cf. \cite[Theorem~11.2.5]{Grochenig1}. 

According to \cite[Corollary~11.2.6]{Grochenig1} 
the topology for $\cS (\rr d)$ can be defined by the collection of seminorms
\begin{equation}\label{eq:seminormsS}
\cS(\rr d) \ni \psi \mapsto \| \psi \|_m := \sup_{z \in \rr {2d}} \eabs{z}^m |V_\fy \psi (z)|, \quad m \in \no,
\end{equation}
for any $\fy \in \cS(\rr d) \setminus 0$. 

The Beurling type Gelfand--Shilov space $\Sigma_\nu^\mu (\rr d)$ is for $\nu,\mu,h > 0$ is defined as
the topological projective limit
\begin{equation*}
\Sigma _\nu^\mu (\rr d) = \bigcap _{h>0} \mathcal S_{\nu,h}^\mu(\rr d)
\end{equation*}
where $\mathcal S_{\nu,h}^\mu(\rr d)$ is the Banach space of smooth functions that have finite 
\begin{equation*}
\nm f{\mathcal S_{\nu,h}^\mu}\equiv \sup_{x \in \rr d, \, \alpha,\beta \in \nn d} \frac {|x^\alpha D^\beta
f(x)|}{h^{|\alpha + \beta |} \alpha !^\nu \, \beta !^\mu}
\end{equation*}
norm \cite{Gelfand2}.   
The space $\Sigma_\nu^\mu (\rr d)$ is a Fr\'echet space with respect to the seminorms $\nm \cdot {\mathcal S_{\nu,h}^\mu}$ for $h > 0$, 
and $\Sigma _\nu^\mu(\rr d)\neq \{ 0\}$ if and only if $\nu + \mu > 1$ \cite{Petersson1}. 

If $\nu + \mu > 1$ the topological dual of $\Sigma _\nu^\mu(\rr d)$ is the space of (Beurling type) Gelfand--Shilov ultradistributions \cite[Section~I.4.3]{Gelfand2}
\begin{equation*}
(\Sigma _\nu^\mu)'(\rr d) =\bigcup _{h>0} (\mathcal S_{\nu,h}^\mu)'(\rr d).
\end{equation*}
The space of ultradistributions $(\Sigma _\nu^\mu)'(\rr d)$ may be equipped with several possibly different topologies \cite{Wahlberg3}. 
In this paper we use exclusively the weak$^*$ topology. 

The Gelfand--Shilov (ultradistribution) spaces enjoy invariance properties, with respect to 
translation, dilation, tensorization, coordinate transformation and (partial) Fourier transformation.
The Fourier transform extends 
uniquely to homeomorphisms on $\mathscr S'(\rr d)$, from $(\Sigma _\nu^\mu)'(\rr d)$ to $(\Sigma _\mu^\nu)'(\rr d)$, and restricts to 
homeomorphisms on $\mathscr S(\rr d)$, 
from $\Sigma _\nu^\mu(\rr d)$ to $\Sigma _\mu^\nu(\rr d)$, and to a unitary operator on $L^2(\rr d)$.

\section{Anisotropic Shubin pseudodifferential calculus}\label{sec:AnisoShubinCalculus}

In this section we retrieve some essential facts from pseudodifferential calculus of anisotropic Shubin
symbols \cite{Rodino4,Wahlberg4}. 

Let $\sigma > 0$. We use the weight function on $(x,\xi) \in T^* \rr d$
\begin{equation}\label{eq:weightanisotrop}
\theta_\sigma(x,\xi) = 1 + |x| + |\xi|^{\frac1\sigma}. 
\end{equation}
For this weight we have the following inequality of Peetre type \cite{Rodino4}. 
If $s \in \ro$ then 
\begin{equation}\label{eq:peetreanisotropic}
\theta_\sigma( x+y, \xi + \eta)^s
\leqs C_{\sigma,s} \theta_\sigma( x, \xi)^{|s|} \theta_\sigma( y, \eta)^s, \quad x, y, \xi, \eta \in \rr d. 
\end{equation}
When $\sigma$ is rational, $\sigma = \frac{k}{m}$, $k,m \in \no \setminus 0$, an alternative weight is
\begin{equation}\label{eq:sobolevweight}
w_{k,m} (x, \xi) 
= \left( 1 + | x |^{2k} + | \xi |^{2 m} \right)^{\frac12}. 
\end{equation}
Note that 
\begin{equation}\label{eq:weightequivalence}
w_{k,m} \asymp \theta_{\sigma}^{k}. 
\end{equation}
The motivation for using $w_{k,m}$ instead of $\theta_{\sigma}^{k}$ is that the former is smooth as opposed to the latter. 

By \cite[Eq.~(3.4)]{Rodino4} we have for $\sigma > 0$
\begin{equation}\label{eq:sobolevweightestimate1}
\eabs{(x,\xi)}^{\min\left( 1, \frac1\sigma\right)} \lesssim \theta_\sigma (x,\xi) \lesssim \eabs{(x,\xi)}^{\max\left( 1, \frac1\sigma\right)}, \quad (x,\xi) \in T^* \rr d, 
\end{equation}
and for  $k,m \in \no \setminus 0$
\begin{equation}\label{eq:sobolevweightestimate2}
\eabs{(x,\xi)}^{\min(k,m)} \lesssim w_{k,m} (x,\xi) \lesssim \eabs{(x,\xi)}^{\max(k,m)}, \quad (x,\xi) \in T^* \rr d. 
\end{equation}

The anisotropic Shubin symbols are defined as follows. 

\begin{defn}\label{def:symbol}
Let $\sigma > 0$ be real and $m \in \ro$. 
The space of ($\sigma$-)anisotropic Shubin symbols $G^{m,\sigma}$ of order $m$ consists of functions $a \in C^\infty(\rr {2d})$ 
that satisfy the estimates
\begin{equation}\label{eq:symbolderivative1}
|\pdd x \alpha \pdd \xi \beta a(x,\xi)|
\lesssim ( 1 + |x| + |\xi|^{\frac1\sigma} )^{m - |\alpha| - \sigma |\beta|}, \quad (x,\xi) \in T^* \rr d, \quad \alpha, \beta \in \nn d. 
\end{equation}
\end{defn}

The space $G^{m,\sigma}$ is a Fr\'echet space with respect to the seminorms on $a \in G^{m,\sigma}$ indexed by $j \in \no$
\begin{equation*}
\| a \|_j = \max_{|\alpha + \beta| \leqs j}
\sup_{(x,\xi) \in \rr {2d} } \theta_\sigma(x,\xi)^{-m + |\alpha| + \sigma |\beta|} \left| \pdd x \alpha \pdd \xi \beta a(x,\xi ) \right|. 
\end{equation*}

If $\sigma = 1$ then $G^{m,\sigma}$ is the space of isotropic Shubin symbols with parameter $\rho = 1$ \cite{Nicola1,Shubin1}.
Recall that the isotropic Shubin symbol of order $m$ and parameter $0 \leqs \rho \leqs 1$, denoted $a\in G_\rho^m$, satisfies  
\begin{equation*}
|\partial_x^\alpha \partial_\xi^\beta a(x,\xi)| \lesssim \langle (x,\xi)\rangle^{m - \rho|\alpha + \beta|}, \quad (x,\xi) \in T^* \rr d, \quad \alpha, \beta \in \nn d.
\end{equation*}
We have $G^{m,\sigma} \subseteq G_\rho^{m_0}$,  
where $m_0 = \max(m, m/\sigma)$ and $\rho = \min(\sigma, 1/\sigma)$,  
and
\begin{equation}\label{eq:symbolintersection}
\bigcap_{m \in \ro} G^{m,\sigma} = \cS(\rr {2d}). 
\end{equation}

The following lemma is a tool for verification of membership in $G^{m,\sigma}$. 

\begin{lem}\label{lem:symbol}
If $m \in \ro$, $\sigma,r > 0$ and $a \in C^\infty(\rr {2d})$ satisfies
\begin{equation}\label{eq:anisotropbound1}
\left| \pdd x \alpha \pdd \xi \beta a (\lambda x, \lambda^\sigma \xi) \right|
\lesssim \lambda^{m-|\alpha| - \sigma |\beta|}, 
\quad (x,\xi) \in T^* \rr d, \quad |(x,\xi)| = r, 
\quad \lambda \geqs 1, \quad \alpha, \beta \in \nn d,  
\end{equation}
then $a \in G^{m,\sigma}$. 
\end{lem}

\begin{proof}
Let $(y,\eta) \in \rr {2d} \setminus \rB_{r}$.
By \cite[Section~3]{Rodino4} $(y,\eta) = (\lambda x, \lambda^\sigma \xi)$ for a unique $(x,\xi) \in \rr {2d}$ such that $|(x,\xi)| = r$ and 
$\lambda \geqs 1$.
Combining
\begin{equation*}
1 + |y| + |\eta|^{\frac1\sigma} = 1 + \lambda ( |x| + |\xi|^{\frac1\sigma} ) \asymp 1+ \lambda
\end{equation*}
with \eqref{eq:anisotropbound1} we obtain for any $\alpha, \beta \in \nn d$
\begin{equation*}
\left| \pdd y \alpha \pdd \eta \beta a (y, \eta) \right|
\lesssim (1+\lambda)^{m -|\alpha| - \sigma |\beta|} 
\lesssim ( 1 + |y| + |\eta|^{\frac1\sigma} )^{m -|\alpha| - \sigma |\beta|}. 
\end{equation*}
The same estimate is trivial for $(y,\eta) \in \rB_{r}$ so
referring to \eqref{eq:symbolderivative1} we may conclude that $a \in G^{m,\sigma}$. 
\end{proof}

\begin{cor}\label{cor:homogeneoussymbol}
If $\sigma > 0$, $m \geqs 0$ and $a \in C^\infty(\rr {2d})$ is anisotropically homogeneous as 
\begin{equation}\label{eq:anisohomsymbol}
a( \lambda x, \lambda^\sigma \xi) = \lambda^{m} a(x,\xi), \quad (x,\xi) \in T^* \rr d, \quad \lambda > 0,  
\end{equation}
then $a \in G^{m, \sigma}$. 
\end{cor}

For $a \in G^{m,\sigma}$ and $\tau \in \ro$ a pseudodifferential operator in the $\tau$-quantization is defined by
\begin{equation}\label{eq:tquantization}
a_\tau(x,D) f(x)
= (2\pi)^{-d}  \int_{\rr {2d}} e^{i \langle x-y, \xi \rangle} a ( (1-\tau) x + \tau y,\xi ) \, f(y) \, \dd y \, \dd \xi, \quad f \in \cS(\rr d),
\end{equation}
when $m < - d \sigma$. The definition extends to $m \in \ro$ if the integral is viewed as an oscillatory integral.
If $\tau = 0$ we get the Kohn--Nirenberg quantization $a_0(x,D) = a(x,D)$ and if $\tau = \frac12$ we have the Weyl quantization $a_{1/2}(x,D) = a^w(x,D)$. 
The Weyl quantization enjoys a simple formal adjoint relation: $a^w(x,D)^* = \overline{a}^w(x,D)$.  
We will use exclusively the Weyl quantization in this paper. 
By \cite[Proposition~3.3 (i)]{Rodino4} the symbol classes $G^{m,\sigma}$ are homeomorphically invariant under change of quantization parameter $\tau \in \ro$, for any $\sigma > 0$ and $m \in \ro$. 
If $a \in G^{m,\sigma}$
then the operator $a^w(x,D)$ acts continuously on $\cS(\rr d)$ and extends uniquely by duality to a continuous operator on $\cS'(\rr d)$
\cite{Rodino4,Shubin1}. 
If $a \in \cS'(\rr {2d})$ then $a^w(x,D)$ extends to a continuous operator $a^w(x,D): \cS(\rr d) \to \cS'(\rr d)$.
If $a \in \cS(\rr {2d})$ then $a^w(x,D)$ is regularizing, in the sense that it is continuous $a^w(x,D): \cS'(\rr d) \to \cS(\rr d)$ with $\cS'(\rr d)$ equipped with the strong topology \cite{Cappiello5}. 

If $a \in \cS'(\rr {2d})$ then 
\begin{equation}\label{eq:wignerweyl}
(a^w(x,D) f, g) = (2 \pi)^{-d} (a, W(g,f) ), \quad f, g \in \cS(\rr d), 
\end{equation}
where the cross-Wigner distribution \cite{Folland1,Grochenig1} is defined as 
\begin{equation*}
W(g,f) (x,\xi) = \int_{\rr d} g (x+y/2) \overline{f(x-y/2)} e^{- i \la y, \xi \ra} \dd y, \quad (x,\xi) \in \rr {2d}. 
\end{equation*}
If $f,g \in \cS(\rr d)$ then $W(g,f) \in \cS (\rr {2d})$. 

Given a sequence of symbols $a_j \in G^{m_j,\sigma}$, $j=1,2,\dots$, such that $m_j \to - \infty$ as $j \to \infty$
we write 
\begin{equation*}
a \sim \sum_{j = 1}^\infty a_j
\end{equation*}
provided that for any $n \geqs 2$
\begin{equation*}
a - \sum_{j = 1}^{n-1} a_j \in G^{\mu_n,\sigma}
\end{equation*}
where $\mu_n = \max_{j \geqs n} m_j$. 
By \cite[Lemma~3.2]{Rodino4} there exists a symbol
$a \in G^{m,\sigma}$ where $m = \max_{j \geqs 1} m_j$ such that $a \sim \sum_{j = 1}^\infty a_j$
under the stated circumstances. 
The symbol $a$ is unique modulo $\cS(\rr {2d})$. 

The bilinear Weyl product $a \wpr b$ of two symbols $a \in G^{m,\sigma}$ and $b \in G^{n,\sigma}$
is defined as the product of symbols corresponding to operator composition: 
$( a \wpr b)^w(x,D) = a^w(x,D) b^w (x,D)$. 
By \cite[Proposition~3.3 (ii)]{Rodino4} the Weyl product is continuous $\wpr: G^{m,\sigma} \times G^{n,\sigma} \to G^{m+n,\sigma}$.
The asymptotic expansion formula for the Weyl product \cite{Hormander1,Shubin1} is
\begin{equation}\label{eq:calculuscomposition1}
a \wpr b(x,\xi) \sim \sum_{\alpha, \beta \geqs 0} \frac{(-1)^{|\beta|}}{\alpha! \beta!} \ 2^{-|\alpha+\beta|}
D_x^\beta \pdd \xi \alpha a(x,\xi) \, D_x^\alpha \pdd \xi \beta b(x,\xi). 
\end{equation}
If $a \in G^{m,\sigma}$ and $b \in G^{n,\sigma}$
then each term in the sum belongs to $G^{m+n -(1+\sigma)|\alpha+\beta|,\sigma}$. 

For $\sigma > 0$ a $\sigma$-conic subset $\Gamma \subseteq T^* \rr d \setminus 0$  
is closed under the operation $T^* \rr d \setminus 0 \ni (x,\xi) \mapsto ( \lambda x, \lambda^\sigma \xi)$
for all $\lambda > 0$. 
By \cite[Definition~3.4 and Lemma~3.5]{Rodino4} (cf. also \cite[Remark~3.4]{Wahlberg4}) it is possible to construct $\sigma$-conic
open subsets of given points in $T^* \rr d \setminus 0$, and corresponding cutoff functions. 

A symbol $a \in G^{m,\sigma}$ is said to be non-characteristic at $z_0 \in T^* \rr d \setminus 0$, 
if 
\begin{equation}\label{eq:noncharacteristic}
|a( x, \xi )| \geqs C \theta_\sigma(x,\xi)^{m}, \quad (x,\xi) \in \Gamma, \quad |(x,\xi)| \geqs R
\end{equation}
for $C,R > 0$, where $\Gamma \subseteq T^* \rr d \setminus 0$ is an open $\sigma$-conic subset containing $z_0$. 
The complement in $T^* \rr d \setminus 0$ of the non-characteristic points is called the characteristic set $\charac_\sigma (a) \subseteq T^* \rr d \setminus 0$. 
It is a closed and $\sigma$-conic subset of $T^* \rr d \setminus 0$. 
This is a particular case of \cite[Definition~3.8]{Rodino4}. 

In most respects the anisotropic Shubin pseudodifferential calculus for the symbol classes $G^{m,\sigma}$ 
with $m \in \ro$ and $\sigma > 0$ works as the isotropic calculus in \cite{Nicola1,Shubin1}. 
In fact \cite[Section~3]{Rodino4} contains the basics of the anisotropic pseudodifferential calculus, 
and by \cite[Lemma~6.3]{Rodino4} and its proof it is possible to construct parametrices
for elliptic symbols. 
Thus if $\sigma > 0$ and $a \in G^{m,\sigma}$ is elliptic in the sense of $\charac_\sigma (a) = \emptyset$, 
that is, 
\begin{equation}\label{eq:ellipticity}
|a( x, \xi )| \geqs C \theta_\sigma(x,\xi)^{m}, \quad (x,\xi) \in \rr {2d} \setminus \rB_R, 
\end{equation}
for $C,R > 0$, then there exists an elliptic symbol $b \in G^{-m,\sigma}$ such that 
\begin{equation*}
a \wpr b = 1 + r_1, \quad b \wpr a = 1 + r_2, 
\end{equation*}
where $r_1, r_2 \in \cS(\rr {2d})$. 

The following definition concerns the anisotropic Gabor wave front set $\WFgs ( u ) \subseteq T^* \rr d \setminus 0$ of $u \in \cS'(\rr d)$ \cite[Definition~4.1]{Rodino4} 
which is important in this paper. 
It is a closed and $\sigma$-conic subset of $T^* \rr d \setminus 0$ well adapted to the anisotropic Shubin calculus. 

\begin{defn}\label{def:WFgs}
Suppose $u \in \cS'(\rr d)$, $\fy \in \cS(\rr d) \setminus 0$, and $\sigma > 0$. 
Then $z_0 = (x_0,\xi_0) \in T^* \rr d \setminus 0$ satisfies $z_0 \notin \WFgs ( u )$
if there exists an open set $U \subseteq T^* \rr d$ such that $z_0 \in U$ and 
\begin{equation}\label{eq:WFgs}
\sup_{(x,\xi) \in U, \ \lambda > 0} \lambda^N |V_\fy u (\lambda x, \lambda^\sigma \xi)| < + \infty \quad \forall N \geqs 0. 
\end{equation}
\end{defn}

If $\sigma = 1$ then $\WFgs ( u ) = \WFg(u)$ that denotes the usual Gabor wave front set \cite{Hormander2,Rodino2}, 
which is isotropic in phase space. The $\sigma$-conic sets are then ordinary cones in $T^* \rr d \setminus 0$, that is sets closed under multiplication with a positive parameter.   

The definition of $\WFgs ( u )$ does not depend on $\fy \in \cS(\rr d) \setminus 0$ \cite[Proposition~4.2]{Rodino4}, and 
 \cite[Proposition~4.3~(i)]{Rodino4} says that 
\begin{equation}\label{eq:WFgFourier}
\WFgs (\wh u) = \J \WF_{\rm g}^{\frac1\sigma} (u), \quad u \in \cS'(\rr d),   
\end{equation}
where 
\begin{equation}\label{eq:Jdef}
\J =
\left(
\begin{array}{cc}
0 & I_d \\
-I_d & 0
\end{array}
\right) \in \rr {2d \times 2d} 
\end{equation}
is the matrix that defines the symplectic group \cite{Folland1}. 

By \cite[Proposition~3.5]{Wahlberg4} we may express the anisotropic Gabor wave front set of $u \in \cS'(\rr d)$ as 
\begin{equation}\label{eq:WFcharacteristic}
\WFgs (u) = \bigcap_{a \in G^{m,\sigma}: \ a^w(x,D) u \in \cS} \charac_{\sigma} (a)
\end{equation}
for any $m \in \ro$. 
For $\sigma > 0$, $m \in \ro$, $a \in G^{m,\sigma}$ and $u \in \cS'(\rr d)$ we have the microlocal and microelliptic inclusions 
\begin{equation*}
\WFgs ( a^w(x,D) u ) \subseteq \WFgs (u) \subseteq \WFgs ( a^w(x,D) u ) \bigcup \charac_\sigma (a)
\end{equation*}
(cf. \cite[Proposition~5.1 and Theorem~6.4]{Rodino4} which are stated slightly more generally). 

At a few occasions we will use anisotropic wave front sets in the Gelfand--Shilov functional framework. 
The Gelfand--Shilov wave front set of $u \in (\Sigma_\nu^\mu)' (\rr d)$ with $\nu + \mu > 1$ is based on the following facts. 
If $\fy \in \Sigma_\nu^\mu (\rr d) \setminus 0$ then 
\begin{equation*}
| V_\fy u (x,\xi)| \lesssim e^{r (|x|^{\frac1\nu} + |\xi|^{\frac1\mu})}
\end{equation*}
for some $r > 0$, and $u \in \Sigma_\nu^\mu (\rr d)$ if and only if 
\begin{equation*}
| V_\fy u (x,\xi)| \lesssim e^{-r (|x|^{\frac1\nu} + |\xi|^{\frac1\mu})}
\end{equation*}
for all $r > 0$. See e.g. \cite[Theorems~2.4 and 2.5]{Toft1}.
The $\nu,\mu$-Gelfand--Shilov wave front set $\WF^{\nu,\mu} (u) \subseteq T^* \rr d \setminus 0$ is defined as follows. 

\begin{defn}\label{def:wavefrontGFst}
Let $\nu,\mu > 0$ satisfy $\nu + \mu > 1$, and suppose $\fy \in \Sigma_\nu^\mu(\rr d) \setminus 0$ and $u \in (\Sigma_\nu^\mu)'(\rr d)$. 
Then $(x_0,\xi_0) \in T^* \rr d \setminus 0$ satisfies $(x_0,\xi_0) \notin \WF^{\nu,\mu} (u)$ if there exists an open set $U \subseteq T^*\rr d \setminus 0$ containing $(x_0,\xi_0)$ such that 
\begin{equation*}
\sup_{\lambda > 0, \ (x,\xi) \in U} e^{r \lambda} |V_\fy u(\lambda^\nu x, \lambda^\mu \xi)| < \infty, \quad \forall r > 0. 
\end{equation*}
\end{defn}

The requested decay is thus exponential rather than superpolynomial as for $\WFgs$. 
The $\nu,\mu$-Gelfand--Shilov wave front set is a closed and $\mu/\nu$-conic subset of $T^* \rr d \setminus 0$ \cite{Rodino3}.

The next result identifies powers of the weight $w_{k,m}$ defined in \eqref{eq:sobolevweight} as anisotropic Shubin symbols. 

\begin{lem}\label{lem:sobolevsymbol}
Let $k,m \in \no \setminus 0$ and $\sigma = \frac{k}{m}$. If $n \in \ro$ then $w_{k,m}^n \in G^{ n k, \sigma}$. 
\end{lem}

\begin{proof}
To simplify notation we write $w = w_{k,m}$. 
It is clear that $w \in C^{\infty} (\rr {2d})$ and that $w$ is positive everywhere. 
We claim that for $\alpha, \beta \in \nn d$ we can write
\begin{equation}\label{eq:symbolderivative2}
\pdd x \alpha \pdd \xi \beta w^n (x, \xi) 
= \left( w (x, \xi)^2 \right)^{\frac{n}{2} - |\alpha + \beta|}
\, p_{\alpha,\beta} (w,x,\xi)
\end{equation}
where $p_{\alpha,\beta}$ are polynomials of the form 
\begin{equation}\label{eq:polyalphabeta}
p_{\alpha,\beta} (w,x,\xi) = \sum_{2 j + \frac{|\gamma|}{k} + \frac{|\kappa|}{m} \leqs \left( 2 - \frac1k \right) |\alpha| + \left( 2 - \frac1m \right)|\beta|} c_{j,\gamma,\kappa} w^{2j} x^\gamma \xi^\kappa 
\end{equation}
with real coefficients $c_{j,\gamma,\kappa}$ for $(j,\gamma,\kappa) \in \no \times \nn d \times \nn d$.

In fact the claim follows from an induction argument with respect to $|\alpha + \beta|$, starting with 
\begin{equation*}
\partial_{x_\ell} \left( \left( w (x, \xi)^2 \right)^{\frac{n}2} \right)
= n k |x|^{2 (k-1)} x_\ell \, \left( w (x, \xi)^2 \right)^{\frac{n}2 - 1}
\end{equation*}
and 
\begin{equation*}
\partial_{\xi_\ell} \left( \left( w (x, \xi)^2 \right)^{\frac{n}2} \right)
= n m |\xi|^{2 (m-1)} \xi_\ell \, \left( w (x, \xi)^2 \right)^{\frac{n}2 - 1}
\end{equation*}
for $1 \leqs \ell \leqs d$. 

Next we estimate a generic monomial in \eqref{eq:polyalphabeta}, using $2 j + \frac{|\gamma|}{k} + \frac{|\kappa|}{m} \leqs \left( 2 - \frac1k \right) |\alpha| + \left( 2 - \frac1m \right)|\beta|$, as
\begin{equation*}
\left| w^{2 j} x^\gamma \xi^\kappa \right|
\leqs w(x,\xi) ^{2 j + \frac{|\gamma|}{k} + \frac{| \kappa |}{m}} 
\leqs w(x,\xi) ^{\left( 2 - \frac1k \right) |\alpha| + \left( 2 - \frac1m \right)|\beta|}. 
\end{equation*}
Inserting into \eqref{eq:symbolderivative2} and exploiting \eqref{eq:weightequivalence} finally give for any $\alpha,\beta \in \nn d$ 
the estimate 
\begin{align*}
\left| \pdd x \alpha \pdd \xi \beta w^n (x, \xi) \right|
& = w(x, \xi)^{n - 2 |\alpha + \beta|}
\, \left| p_{\alpha,\beta} (w,x,\xi) \right| \\
& \lesssim w(x, \xi)^{n - \frac1k  |\alpha| - \frac1m  |\beta|} \\
& \asymp \theta_\sigma (x,\xi)^{n k - |\alpha| - \sigma |\beta| }. 
\end{align*}
\end{proof}

Suppose $\sigma > 0$ is rational, that is $\sigma = \frac{k}{m}$ with $k,m \in \no \setminus 0$. 
In \cite[Proposition~4.2]{Chatzakou1} the authors identify the symbol class 
$a \in G^{n,\sigma}$ with $n \in \ro$
as the Weyl--H\"ormander symbol class \cite[Chapter~18.4]{Hormander1}
\begin{equation}\label{eq:weylhormandersymbol}
G^{n,\sigma} = S( h_g^{-\frac{n}{1 + \sigma}}, g) 
\end{equation}
defined by the metric
\begin{equation*}
g = \frac{\dd x^2}{\left( 1 + |x|^{2k} + |\xi|^{2m} \right)^{\frac1k}} + \frac{\dd \xi^2}{ \left(1 + |x|^{2k} + |\xi|^{2m} \right)^{\frac1m} } 
\end{equation*}
and the weight $h_g^{-\frac{n}{1 + \sigma}}$. 
Here $h_g$ is the so called Planck function associated to $g$ \cite{Chatzakou1,Hormander1}. 
The Planck function is according to \cite[Remark~2.4]{Chatzakou1}
\begin{equation}\label{eq:planckfunction}
h_g (x,\xi) = \left( 1 + |x|^{2k} + |\xi|^{2m} \right)^{- \frac12 \left(\frac1k + \frac1m \right)}
= \left( w_{k,m} (x,\xi) \right)^{- \left(\frac1k + \frac1m \right)}. 
\end{equation}

From this two conclusions follows: First we observe that $h_g$ satisfies the so called uncertainty principle
\begin{equation}\label{eq:uncertaintyprinciple}
h_g (x,\xi) \leqs 1 \quad \forall (x,\xi) \in T^* \rr d, 
\end{equation}
and secondly by Lemma \ref{lem:sobolevsymbol} we have $h_g \in G^{-1-\sigma,\sigma}$.

\section{Globally anisotropic Shubin--Sobolev spaces, localization operators and a sharp G{\aa}rding inequality}
\label{sec:ShubinSobolev}

In this paper we will often use the following parametrized family of Hilbert modulation spaces. 
These spaces also have an independent interest. Proposition \ref{prop:continuitySobolev} complements the anisotropic Shubin pseudodifferential calculus in \cite{Rodino4}. 

\begin{defn}\label{def:Sobolevaniso}
Let $\fy \in \cS(\rr d) \setminus 0$. 
The anisotropic Shubin--Sobolev modulation space $M_{\sigma,s} (\rr d)$ with anisotropy parameter $\sigma > 0$ and order $s \in \ro$ is the Hilbert subspace of $\cS'(\rr d)$
defined by the norm
\begin{equation}\label{eq:SSmodnorm}
\| u \|_{M_{\sigma,s}} = \left( \iint_{\rr {2 d}} |V_\fy u (x,\xi)|^2 \, \theta_\sigma (x,\xi)^{2 s} \, \dd x \, \dd \xi \right)^{\frac12}. 
\end{equation}
\end{defn}

For any $\sigma > 0$ we have $M_{\sigma,0} (\rr d) = L^2(\rr d)$ \cite{Grochenig1}, 
and $M_{\sigma,s_1} (\rr d) \subseteq M_{\sigma,s_2}(\rr d)$ is a continuous inclusion when $s_1 \geqs s_2$. 
It holds
\begin{equation}\label{eq:SSSchwartz}
\cS(\rr d) = \bigcap_{s \in \ro} M_{\sigma,s} (\rr d), \quad \cS'(\rr d) = \bigcup_{s \in \ro} M_{\sigma,s} (\rr d), 
\end{equation}
and $\{ \| \cdot \|_{M_{\sigma,s}}, s \geqs 0\}$ is a family of seminorms that defines the Fr\'echet space topology on $\cS(\rr d)$ \cite{Grochenig1}. 

The next continuity result is a natural generalization of the isotropic Shubin calculus. 
More precisely it generalizes \cite[Proposition~1.5.5]{Nicola1} and \cite[Theorem~25.2]{Shubin1}. 

\begin{prop}\label{prop:continuitySobolev}
Let $\sigma > 0$ and $m,s \in \ro$. 
If $a \in G^{m,\sigma}$ then 
\begin{equation}\label{eq:ShubinSobolevCont}
a^w(x,D) : M_{\sigma,s + m} (\rr d) \to M_{\sigma,s} (\rr d)
\end{equation}
is continuous. 
\end{prop}

\begin{proof}
By a small modification of the proof of \cite[Proposition~3.2]{Cappiello3} 
it follows that $a \in G^{m,\sigma}$ if and only if 
\begin{equation}\label{eq:symbolcharacterization}
\left| \pdd {z_1} \alpha \pdd {z_2} \beta \cT_\fy a( z ,\zeta ) \right|
\lesssim \theta_\sigma (z)^{m - |\alpha| - \sigma |\beta|} \eabs{\zeta}^{-N}, \quad z, \zeta \in \rr {2d}, \quad \alpha, \beta \in \nn d, \quad N \geqs 0, 
\end{equation}
where $z = (z_1,z_2)$ with $z_1, z_2 \in \rr d$, $g \in \cS(\rr {2d}) \setminus 0$, and where $\cT_\fy u$ is defined by 
\begin{equation*}
\cT_\fy u(x,\xi)=(2\pi)^{-\frac{d}{2}} (u,T_x M_{\xi} \fy) = e^{i \la x, \xi \ra} V_\fy u(x,\xi), 
\quad x, \xi \in \rr d, 
\end{equation*}
for $u \in \cS'(\rr d)$ and $\fy \in \cS(\rr d) \setminus 0$. 
In fact in the original proof \cite{Cappiello3} we only have to replace the weight $\eabs{\cdot}$ used there by $\theta_\sigma$, 
take into account the behavior with respect to derivatives of $a \in G^{m,\sigma}$ with respect to $z_1$ and $z_2$ respectively, 
and use \eqref{eq:peetreanisotropic}. 

Let $\fy \in \cS(\rr d) \setminus 0$ and set $\Phi = W(\fy, \fy) \in \cS(\rr {2d}) \setminus 0$. 
If $u \in \cS'(\rr d)$ then by \cite[Eq.~(5.3)]{Rodino4} we have 
\begin{equation}\label{eq:STFTWeylop}
|V_\fy (a^w(x,D) u) (z)|
\lesssim \int_{\rr {2d}}  |V_\fy u(z-w)| \, \left| V_\Phi a \left( z-\frac{w}{2}, \J w \right) \right| \, \dd w.
\end{equation}
We obtain from \eqref{eq:symbolcharacterization}, \eqref{eq:peetreanisotropic} and 
\eqref{eq:sobolevweightestimate1}
the estimates 
\begin{align*}
\left| V_\Phi a \left( z - \frac{w}{2},  \J w \right) \right|
& \lesssim \theta_\sigma( z-w )^{m} \theta_\sigma (w)^{|m|} \eabs{ w }^{-N} \\
& \lesssim \theta_\sigma ( z-w )^{m} \eabs{ w }^{- \left( N - |m| \max(1, \frac1\sigma) \right)}, \quad z, w \in \rr {2d}, \quad N \geqs 0. 
\end{align*}
Combining this with \eqref{eq:STFTWeylop}, Minkowski's inequality and again \eqref{eq:peetreanisotropic} yields
\begin{align*}
\| a^w (x,D) u \|_{M_{\sigma,s}}
& =  \left\| V_\fy (a^w (x,D) u) \, \theta_\sigma^s \right\|_{L^2 (\rr {2d})} \\
& \lesssim \left\|  \int_{\rr {2d}}  |V_\fy u(\cdot-w)| \, \left| V_\Phi a \left( \cdot-\frac{w}{2}, \J w \right) \right| \, \dd w \, \theta_\sigma (\cdot)^s \right\|_{L^2 (\rr {2d})} \\
& \lesssim \left\|  \int_{\rr {2d}}  |V_\fy u(\cdot-w)| \,  \theta_\sigma( \cdot -w )^{m + s} \eabs{ w }^{- \left( N - ( |m| + |s| ) \max(1, \frac1\sigma) \right)} \, \dd w \right\|_{L^2 (\rr {2d})} \\
& \lesssim \left\| V_\fy u \,  \theta_\sigma^{m + s} \right\|_{L^2 (\rr {2d})} \\
& = \| u \|_{M_{\sigma,m + s}}
\end{align*}
provided $N \geqs 0$ is sufficiently large. 
\end{proof}

Let $\fy \in \cS(\rr d)$ satisfy $\| \fy \|_{L^2} = 1$. 
A localization operator $A_a$ with symbol $a \in \cS'(\rr {2d})$ is defined as 
\begin{equation}\label{eq:locop}
(A_a f, g) = (a V_\fy f, V_\fy g) = (V_\fy^* a V_\fy f, g), \quad f,g \in \cS(\rr d), 
\end{equation}
that is $A_a = V_\fy^* a V_\fy$. Then $A_a: \cS(\rr d) \to \cS'(\rr d)$ is continuous. 
We will assume that $\fy$ is a Gaussian. 

By \cite[Theorem~1.1]{Grochenig2} we have 
for any $\sigma > 0$ and $s \in \ro$
\begin{equation}\label{eq:locopmodspace}
\| A_{\theta_\sigma^s} u \|_{L^2(\rr d)}
 \asymp 
 \| u \|_{M_{\sigma,s}(\rr d)}
\end{equation}
which means that $A_{\theta_\sigma^s}: M_{\sigma,s}(\rr d) \to L^2(\rr d)$ 
is an isometry. 

If $a \in \cS'(\rr {2d})$ and $\fy$ is a Gaussian on $\rr d$ we have $A_a = b^w(x,D)$ where $b = a * \psi$ 
with $\psi$ a Gaussian on $\rr {2d}$ \cite[Proposition~1.7.9]{Nicola1}. 
If $\sigma > 0$ and $a \in G^{m, \sigma}$ then also $b \in G^{m, \sigma}$, and 
$a$ real-valued implies that also $b$ is real-valued \cite[Theorem~1.7.10]{Nicola1}.

In Section \ref{sec:solutionsaniso} we will need the following inequality of sharp G{\aa}rding type. 

\begin{lem}\label{lem:Garding}
Let $k,m \in \no \setminus 0$ and $\sigma = \frac{k}{m}$. 
If $a \in G^{2\left(1 + \sigma \right),\sigma}$
and $a \geqs 0$ then there exists $c > 0$ such that 
\begin{equation}\label{eq:gardinginequality}
( a^w(x,D) f, f) \geqs - c \| f \|_{L^2}^2, \quad f \in \cS(\rr d).
\end{equation}
\end{lem}

\begin{proof}
By \eqref{eq:weylhormandersymbol} we have $G^{2\left(1 + \sigma \right),\sigma} = S( h_g^{-2}, g)$, 
where the Planck function $h_g$ is defined by \eqref{eq:planckfunction} and satisfies the uncertainty principle 
\eqref{eq:uncertaintyprinciple}. 
The conclusion is now a consequence of the Fefferman--Phong inequality \cite[Theorem~18.6.8]{Hormander1}. 
\end{proof}

\section{Propagation of anisotropic Gabor wave front sets for evolution equations of Airy--Schr\"odinger type}
\label{sec:airyschrodinger}

In this section we consider the evolution equation
\begin{equation}\label{eq:airyschrodeq}
\left\{
\begin{array}{rl}
\partial_t u(t,x) + i \left( p(D_x) + \la v, x \ra \right)u (t,x) & = 0, \quad x \in \rr d, \quad t \in \ro, \\
u(0,\cdot) & = u_0.
\end{array}
\right.
\end{equation}
Here $v = (v_1, \dots, v_d) \in \rr d$ is a vector with nonzero entries: $v_j \neq 0$, $1 \leqs j \leqs d$, and $p: \rr d \to \ro$ is a polynomial with real coefficients of order $m \geqs 2$ 
which is a sum of one variable polynomials, that is
\begin{equation}\label{eq:polynomialsum}
p (\xi) = \sum_{j=1}^d p_j(\xi_j), \quad \xi = ( \xi_1, \xi_2, \dots, \xi_d) \in \rr d, 
\end{equation}
where 
\begin{equation}\label{eq:polynomialonevar}
p_j (\xi_j) = \sum_{k=0}^{m_j} c_{j,k} \xi_j^k, \quad c_{j,k} \in \ro, \quad c_{j,m_j} \neq 0, 
\end{equation}
and $\max_{j=1}^d \deg p_j = \max_{j=1}^d m_j = m$. 
The principal part of $p$ is 
\begin{equation}\label{eq:principalpart}
P_m (\xi) = \sum_{j \in \{1, \dots,d \} : \ m_j = m} c_{j,m} \xi_j^m.
\end{equation}

We say that the equation \eqref{eq:airyschrodeq} is of Airy--Schr\"odinger type, since when $d = 1$ a particular case of the Hamiltonian is the operator
\begin{equation*}
a(x,D) = \frac{\dd^2}{\dd x^2} - x
\end{equation*}
which defines the Airy equation $a(x,D) f = 0$. 
This equation is satisfied by the Airy function \cite[Chapter~7.6]{Hormander1}. 

First we deduce the explicit solution $u(t,x) = \cK_t u_0 (x)$ to \eqref{eq:airyschrodeq} defined by the propagator $\cK_t$, 
and in particular an expression for the Schwartz kernel $K_t$ of $\cK_t$ for each $t \in \ro$. 
Let $q_j$ be primitive polynomials of $p_j$: 
\begin{equation}\label{eq:primitivepolynomial}
q_j' = p_j, \quad 1 \leqs j \leqs d. 
\end{equation}
If $u_0 \in \cS(\rr d)$ then the solution to \eqref{eq:airyschrodeq} is given by
\begin{equation}\label{eq:solution1}
\begin{aligned}
u(t,x) 
& = (2 \pi)^{-\frac{d}{2}} \int_{\rr d} e^{i \left( \la x ,\xi - v t \ra  + \sum_{j=1}^d v_j^{-1}  \left( q_j(\xi_j - t v_j )  - q_j( \xi_j) \right) \right)}  \wh u_0(\xi) \dd \xi \\
& = e^{- i t \la x, v \ra } \cF^{-1} \left(  e^{i \fy_t }  \wh u_0 \right) (x) 
= M_{- t v} \cF^{-1} \left(  e^{i \fy_t }  \wh u_0 \right) (x)
\end{aligned}
\end{equation}
where 
\begin{equation}\label{eq:phasefunction}
\fy_t (\xi) = \sum_{j=1}^d v_j^{-1}  \left( q_j(\xi_j - t v_j)  - q_j( \xi_j) \right). 
\end{equation}
This can be confirmed by insertion of \eqref{eq:solution1} into \eqref{eq:airyschrodeq}. 

The solution operator 
\begin{equation}\label{eq:solutionoperator}
\cK_t f = M_{- t v} \cF^{-1} \left(  e^{i \fy_t }  \wh f \right)
\end{equation}
is unitary on $L^2(\rr d)$, and since $\fy_{-t}(\xi) = - \fy_t(\xi + t v )$ we obtain for $f,g \in \cS(\rr d)$
\begin{align*}
( \cK_t f, g) 
& = \left( \wh f, e^{- i \fy_t} \cF \left( M_{t v} g \right) \right)  
= \left( \wh f, e^{- i \fy_t} T_{t v} \wh g \right) 
= \left( \wh f, T_{t v} \left( e^{- i \fy_t(\cdot + t v )} \wh g \right) \right) \\
& = \left( f, \cF^{-1} \left( T_{t v} \left( e^{i \fy_{-t}} \wh g \right) \right) \right)
= \left( f, M_{t v}  \cF^{-1} \left( e^{i \fy_{-t}} \wh g \right) \right) \
= (f, \cK_{-t} g)
\end{align*}
so $\cK_t^* = \cK_{-t} = \cK_t^{-1}$.
If $t_1, t_2 \in \ro$ then 
\begin{equation*}
\fy_{t_1}(\xi-t_2 v) + \fy_{t_2} (\xi) = \fy_{t_1 + t_2} (\xi)
\end{equation*}
which gives
\begin{align*}
\cK_{t_1} \cK_{t_2} f
& = M_{- t_1 v} \cF^{-1} \left(  e^{i \fy_{t_1} } \cF \left( M_{- t_2 v} \cF^{-1} \left(  e^{i \fy_{t_2} }  \wh f \right) \right) \right) \\
& = M_{- t_1 v} \cF^{-1} \left(  e^{i \fy_{t_1} } T_{- t_2 v}  \left(  e^{i \fy_{t_2} }  \wh f \right) \right) \\
& = M_{- t_1 v} \cF^{-1} \left(  T_{- t_2 v}  \left( e^{i \left( \fy_{t_1}(\cdot -t_2 v) + \fy_{t_2} \right)}  \wh f \right) \right) \\
& = M_{- (t_1 + t_2) v} \cF^{-1} \left(  e^{i \fy_{t_1 + t_2} }  \wh f \right)  
= \cK_{t_1 + t_2} f
\end{align*}
so the map $\ro \ni t \mapsto \cK_t$ is in fact a one-parameter group of unitary operators. 

The Schwartz kernel of the solution operator $\cK_t$ is
\begin{equation}\label{eq:schwartzkernel}
\begin{aligned}
K_t (x,y) & = (2 \pi)^{-d} \int_{\rr d} e^{i \left( \la x ,\xi - v t \ra - \la y, \xi \ra  + \fy_t(\xi)  \right)} \dd \xi \\
& = (2 \pi)^{-\frac{d}{2}}  e^{- i t \la x, v \ra } 
\cF^{-1} \left( e^{ i \fy_t }\right) (x-y) \\
& = (2 \pi)^{-\frac{d}{2}}  e^{- i t \la x, v \ra } \left( 1 \otimes \cF^{-1} 
e^{i \fy_t} \right) \circ \kappa^{-1}(x,y) \in \cS'(\rr {2d})
\end{aligned}
\end{equation}
where $\kappa \in \rr {2d \times {2d}}$ is the matrix defined by $\kappa(x,y) = (x+\frac{y}{2}, x - \frac{y}{2})$ for $x,y \in \rr d$. 
We note that $\cK_t$ acts continuously on $\cS(\rr d)$ for any $t \in \ro$, 
and extends uniquely to a continuous linear operator on $\cS' (\rr d)$ by
\begin{equation*}
( \cK_t u, \fy) := (u, \cK_{-t} \fy ), \quad u \in \cS'(\rr d), \quad \fy \in \cS(\rr d).  
\end{equation*}

For each $1 \leqs j \leqs d$ we have 
\begin{align*}
q_j(\xi_j - t v_j)  - q_j( \xi_j)
& = \sum_{k=0}^{m_j} \frac{c_{j,k}}{k+1}  \left( (\xi_j - t v_j)^{k+1} - \xi_j^{k+1} \right) \\
& = \sum_{k=0}^{m_j} \left( - c_{j,k} t v_j \xi_j^k + \frac{c_{j,k}}{k+1}  \sum_{n=2}^{k+1} \binom{k+1}{n} (-t v_j)^{n} \xi_j^{k+1-n} \right). 
\end{align*}
Hence the phase function $\fy_t (\xi) $
is a polynomial of order $m$ with highest order term
\begin{equation*}
\fy_{t,m} (\xi) = - t \sum_{j \in \{1, \dots,d \} : \ m_j = m} c_{j,m} \xi_j^m = - t P_m(\xi). 
\end{equation*}

We may now give a result which generalizes a particular case of \cite[Theorem~5.1]{Wahlberg4}. 
More precisely, in the quoted result the polynomial $p$ is arbitrary with real coefficients, whereas here we assume the particular ``separable'' form 
\eqref{eq:polynomialsum}. On the other hand, in Theorem \ref{thm:propagationAiry1} below we allow a vector $v \in \rr d$ with non-zero entries. 
The result uses the Hamilton flow corresponding to the principal part $P_m(\xi)$ of the polynomial $p(\xi)$, that is  
\begin{equation}\label{eq:hamiltonflow1}
\chi_t (x, \xi)
= (x + t \nabla P_{m} (\xi), \xi), \quad t \in \ro, \quad (x, \xi) \in T^* \rr d \setminus 0. 
\end{equation}

\begin{thm}\label{thm:propagationAiry1}
Let $p$ be a polynomial with real coefficients defined by \eqref{eq:polynomialsum}, \eqref{eq:polynomialonevar}, 
of order $m = \max_{j=1}^d \deg p_j \geqs 2$, 
with principal part $P_m$ defined by \eqref{eq:principalpart}. 
Denote the Hamilton flow of $P_m(\xi)$ as in \eqref{eq:hamiltonflow1}. 
Suppose $\cK_t: \cS' (\rr d) \to \cS' (\rr d)$ is the solution operator for the evolution equation \eqref{eq:airyschrodeq}, with Schwartz kernel \eqref{eq:schwartzkernel}
where $\fy_t$ is defined by \eqref{eq:primitivepolynomial} and \eqref{eq:phasefunction}. 
Then 
\begin{align}
\WFgs( \cK_t u)  & = \chi_t  \left( \WFgs (u) \right), \quad t \in \ro, \quad u \in \cS'(\rr d), \quad \sigma = \frac1{m-1}, \label{eq:propagationG1} \\
\WFgs( \cK_t u)  & = \WFgs (u), \quad t \in \ro, \quad u \in \cS'(\ro), \quad \sigma < \frac1{m-1}. \label{eq:propagationG2}
\end{align}
\end{thm}

\begin{proof}
By \cite[Theorems~7.1 and 7.2]{Rodino4} we have 
\begin{equation}\label{eq:WFGchirp1}
\begin{aligned}
\WF_{\rm g}^{m-1} \left( e^{i \fy_t } \right)
& \subseteq \{ \left( x, \nabla \fy_{t,m} (x) \right) \in \rr {2 d}: x \in \rr d \setminus 0 \} \\
& = \{ \left( x, - t \nabla P_{m} (x) \right) \in \rr {2 d}: x \in \rr d \setminus 0 \}, \\
\WFgs \left(e^{i \fy_t } \right) 
& \subseteq \left( \rr d \setminus 0 \right) \times \{0\}, \quad \sigma > m-1. 
\end{aligned}
\end{equation}
Combining \eqref{eq:WFGchirp1} with \cite[Eq.~(4.6) and Proposition~4.3 (i)]{Rodino4}, cf. \eqref{eq:WFgFourier}, gives 
\begin{equation}\label{eq:WFGchirp2}
\begin{aligned}
\WF_{\rm g}^{\frac1{m-1}} \left( \cF^{-1} e^{i \fy_t } \right) 
& \subseteq \{ \left( t \nabla P_m (x), x) \right) \in \rr {2d}: x \in \rr d \setminus 0 \}, \\
\WFgs \left( \cF^{-1} e^{i \fy_t } \right)
& \subseteq \{0\} \times \left( \rr d \setminus 0 \right) , \quad \sigma < \frac1{m-1}. 
\end{aligned}
\end{equation}

Now \eqref{eq:schwartzkernel}, \cite[Corollary~5.2 and Proposition~4.3~(ii)]{Rodino4}, \cite[Proposition~3.2]{Wahlberg4}, \cite[Proposition~5.3 (iii)]{Rodino4} and \eqref{eq:WFGchirp2} yield if $\sigma = \frac1{m-1}$
\begin{align*}
& \WFgs (K_t)
= \WFgs ( \left( 1 \otimes \cF^{-1} e^{i \fy_t} \right) \circ \kappa^{-1} ) \\
& = 
\left( 
\begin{array}{cc}
  \kappa & 0 \\
  0 & \kappa^{-T}
  \end{array}
\right) 
\WFgs \left( 1 \otimes \cF^{-1} e^{i \fy_t} \right) \\
& \subseteq \{ ( \kappa( x_1, x_2), \kappa^{-T}(\xi_1, \xi_2) ) \in T^* \rr {2d}: \\
& \qquad \qquad \qquad \qquad \qquad  (x_1, \xi_1) \in \WFgs (1) \cup \{ 0 \}, \  (x_2, \xi_2) \in \WFgs ( \cF^{-1} e^{i \fy_t} ) \cup \{ 0 \} \} \setminus 0 \\
& = \{ ( \kappa( x_1, t \nabla P_{m} (x_2) ), \kappa^{-T} (0, x_2) \in T^* \rr {2d}: \ x_1, x_2 \in \rr d \} \setminus 0 \\
& = \left\{ \left( x_1 +t \frac12 \nabla P_{m} (x_2), x_1 - t \frac12 \nabla P_{m} (x_2) , x_2, - x_2 \right) \in T^* \rr {2d}: \ x_1, x_2 \in \rr d \right\} \setminus 0 \\
& = \left\{ \left( x_1 + t  \nabla P_{m} (x_2), x_1, x_2, - x_2 \right) \in T^* \rr {2d}: \ x_1, x_2 \in \rr d \right\} \setminus 0. 
\end{align*}

Since $m \geqs 2$ we have $\nabla P_{m} (0) = 0$. 
Hence $\WFgs (K_t)$ does not contain points of the form $(x, 0, \xi, 0)$ nor of the form $(0, x, 0, -\xi)$ for any $(x,\xi) \in T^* \rr d \setminus 0$. 
We may therefore apply \cite[Theorem~4.4]{Wahlberg4} which gives for $u \in \cS'(\rr d)$
\begin{equation}\label{eq:propinclusion1}
\begin{aligned}
\WFgs( \cK_t u) 
& \subseteq \WFgs (K_t)' \circ \WFgs (u) \\
& = \{ (x,\xi) \in T^* \rr d: \ \exists (y,\eta) \in \WFgs (u), \ (x,y, \xi, - \eta) \in \WFgs (K_t) \} \\
& \subseteq \{ ( x_1 + t \nabla P_{m} (x_2)  , x_2) \in T^* \rr d: \ (x_1,x_2) \in \WFgs (u) \} \\
& = \chi_t  \left( \WFgs (u) \right). 
\end{aligned}
\end{equation}

Since $\cK_t^{-1} = \cK_{-t}$ and $\chi_t^{-1} = \chi_{-t}$ we may strengthen \eqref{eq:propinclusion1} into
\begin{equation*}
\WFgs( \cK_t u)  = \chi_t  \left( \WFgs (u) \right), \quad t \in \ro, \quad u \in \cS'(\rr d), \quad \sigma = \frac1{m-1}.
\end{equation*}
We have proved \eqref{eq:propagationG1}. 

Likewise if $\sigma < \frac1{m-1}$ then again 
\eqref{eq:schwartzkernel}, \cite[Corollary~5.2 and Proposition~4.3~(ii)]{Rodino4}, \cite[Proposition~3.2]{Wahlberg4}, \cite[Proposition~5.3 (iii)]{Rodino4} and \eqref{eq:WFGchirp2} yield
\begin{align*}
\WFgs (K_t)
& \subseteq \{ ( \kappa( x_1, x_2), \kappa^{-T}(\xi_1, \xi_2) ) \in T^* \rr {2 d}: \\
& \qquad \qquad \qquad \qquad (x_1, \xi_1) \in \WFgs (1) \cup \{ 0 \},  
\ (x_2, \xi_2) \in \WFgs ( \cF^{-1} e^{i \fy_t} ) \cup \{ 0 \} \} \setminus 0 \\
& \subseteq \{ ( \kappa( x_1, 0 ), \kappa^{-T}(0, x_2) \in T^* \rr {2 d}: \ x_1, x_2 \in \rr d \} \setminus 0 \\
& = \left\{ \left( x_1 , x_1, x_2, - x_2 \right) \in T^* \rr {2d}: \ x_1, x_2 \in \rr d \right\} \setminus 0. 
\end{align*}
Again \cite[Theorem~4.4]{Wahlberg4} gives
\begin{equation*}
\WFgs( \cK_t u)  = \WFgs (u), \quad t \in \ro, \quad u \in \cS'(\rr d), \quad \sigma < \frac1{m-1}, 
\end{equation*}
which proves \eqref{eq:propagationG2}. 
\end{proof}

\begin{rem}\label{rem:comparison1}
The conclusion from \eqref{eq:propagationG1} and \eqref{eq:propagationG2} is that the propagation of singularities for 
the equation \eqref{eq:airyschrodeq} works exactly 
as when $v = 0$, as described in \cite[Theorem~5.1]{Wahlberg4}. 
The Hamiltonian in \eqref{eq:airyschrodeq} is $a(x,\xi) = p(\xi) + \la v, x \ra$, 
but the propagation of singularities follows the Hamiltonian flow of $P_m(\xi)$. 
Note that $a_0(x,\xi) = P_m(\xi)$ satisfies the anisotropic homogeneity 
\begin{equation*}
a_0( \lambda x, \lambda^\sigma \xi) = \lambda^{1 + \sigma} a_0(x,\xi), 
\quad (x,\xi) \in T^* \rr d, \quad \lambda > 0,  
\end{equation*}
if $\sigma = \frac1{m-1}$, so $a_0 \in G^{1 + \sigma,\sigma}$ according to Corollary \ref{cor:homogeneoussymbol}. 

If we decompose the polynomial $p$ as 
\begin{equation*}
p(\xi) = P_m(\xi) + \sum_{j=0}^{m-1} P_j (\xi)
\end{equation*}
where each term $P_j (\xi)$ is homogeneous of degree $j$ for $0 \leqs j \leqs m$, 
then each term $P_j$ satisfies
\begin{equation*}
P_j( \lambda^\sigma \xi) = \lambda^{\frac{j}{m-1}} P_j(\xi), 
\quad \xi \in \rr d, \quad \lambda > 0,  \quad 0 \leqs j \leqs m. 
\end{equation*}
Thus $a-a_0 = \sum_{j=0}^{m-1} P_j + \la v, x \ra = \sum_{j=0}^{m-1} b_j $ with terms $b_j$, considered as functions on $(x,\xi) \in T^* \rr d$,
of homogeneities
\begin{equation}\label{eq:loworderhomog}
b_j( \lambda x, \lambda^\sigma \xi) = \lambda^{\frac{j}{m-1}} b_j(x,\xi), 
\quad (x,\xi) \in T^* \rr d, \quad \lambda > 0, \quad 0 \leqs j \leqs m-1. 
\end{equation}

The terms $b_j$ have all smaller order $\frac{j}{m-1} = j \sigma$ of anisotropic homogeneity than the principal part $a_0 = P_m$
which has order $1 + \sigma = m \sigma$, and which governs the propagation of singularities.
Thus one may see $a - a_0$ as lower order perturbations of the Hamiltonian
that do not affect propagation of singularities. 
Note that the Hamiltonian term $\la v, x \ra$ satisfies \eqref{eq:loworderhomog} with $j = m-1$. 
\end{rem}

As a complementary result we formulate a version of Theorem \ref{thm:propagationAiry1} in the framework of 
Beurling type Gelfand--Shilov spaces $\Sigma_\nu^\mu (\rr d)$ for $\nu + \mu > 1$
and their dual ultradistribution spaces $\left( \Sigma_\nu^\mu \right)'(\rr d)$. 

\begin{thm}\label{thm:propagationAiry2}
Let $p$ be a polynomial with real coefficients defined by \eqref{eq:polynomialsum}, \eqref{eq:polynomialonevar}, 
of order $m = \max_{j=1}^d \deg p_j \geqs 2$, 
with principal part $P_m$ defined by \eqref{eq:principalpart}. 
Denote the Hamilton flow of $P_m(\xi)$ as in \eqref{eq:hamiltonflow1}. 
Suppose $\cK_t: \cS (\rr d) \to \cS (\rr d)$ is the 
solution operator for the evolution equation \eqref{eq:airyschrodeq}, with Schwartz kernel \eqref{eq:schwartzkernel}
where $\fy_t$ is defined by \eqref{eq:primitivepolynomial} and \eqref{eq:phasefunction}. 

If $\nu \geqs \mu(m-1) > 1$ then $\cK_t$ is continuous on $\Sigma_\nu^\mu(\rr d)$, extends uniquely to a continuous linear operator on $\left( \Sigma_\nu^\mu \right)'(\rr d)$, and 
\begin{align}
\WF^{\nu,\mu} ( \cK_t u)  & = \chi_t  \left( \WF^{\nu,\mu} (u) \right), \quad t \in \ro, \quad u \in \left( \Sigma_\nu^\mu \right)'(\rr d), \quad \nu = \mu(m-1) > 1, \label{eq:propagationGS1} \\
\WF^{\nu,\mu} ( \cK_t u)  & = \WF^{\nu,\mu} (u), \quad t \in \ro, \quad u \in \left( \Sigma_\nu^\mu \right)'(\rr d), \quad \nu > \mu(m-1) > 1. \label{eq:propagationGS2} 
\end{align}
\end{thm}

\begin{proof}
In the Gelfand--Shilov functional framework we have, similar to \eqref{eq:WFGchirp1}, 
by \cite[Theorems~6.1 and 6.2]{Wahlberg3} if $\nu > \frac1{m-1}$
\begin{equation}\label{eq:WFGSchirp1}
\begin{aligned}
\WF^{\nu,\nu(m-1)} \left( e^{i \fy_t }  \right)
& \subseteq \{ \left( x, - t \nabla P_{m}(x) \right) \in \rr {2d}: x \in \rr d \setminus 0 \}, \\
\WF^{\nu,\mu} \left( e^{i \fy_t } \right)
& \subseteq \left( \rr d \setminus 0 \right) \times \{0\}, \quad \mu > \nu(m-1). 
\end{aligned}
\end{equation}

As before \cite[Eq.~(3.8) and Proposition~3.6~(i)]{Rodino3} 
give
\begin{equation}\label{eq:WFGSchirp2}
\begin{aligned}
\WF^{\nu(m-1),\nu} \left( \cF^{-1} e^{i \fy_t } \right)
& \subseteq \{ \left( t \nabla P_{m}(x), x \right) \in \rr {2d}: x \in \rr d \setminus 0 \}, \\
\WF^{\mu,\nu} \left( \cF^{-1}  e^{i \fy_t } \right)
& \subseteq \{0\} \times \left( \rr d \setminus 0 \right), \quad \mu > \nu(m-1). 
\end{aligned}
\end{equation}

From \cite[Proposition~4.5]{Wahlberg3} and \cite[Proposition~3.6~(ii), Corollary~6.4 and Proposition~7.1~(iii)]{Rodino3} 
we obtain if $\nu = \mu(m-1) > 1$ 
\begin{equation*}
\WF^{\nu,\mu} (K_t)
\subseteq \left\{ \left( x_1 + t \nabla P_{m}(x_2) , x_1, x_2, - x_2 \right) \in T^* \rr {2d}: \ x_1, x_2 \in \rr d \right\} \setminus 0,  
\end{equation*}
and if $\nu > \mu(m-1) > 1$
\begin{equation*}
\WF^{\nu,\mu} (K_t)
\subseteq \left\{ \left( x_1, x_1, x_2, - x_2 \right) \in T^* \rr {2d}: \ x_1, x_2 \in \rr d \right\} \setminus 0. 
\end{equation*}

At this point \cite[Theorem~5.5]{Wahlberg3} yields the following two final conclusions: 
If $\nu \geqs \mu(m-1) > 1$ then 
$\cK_t$ is continuous on $\Sigma_\nu^\mu(\rr d)$ and extends uniquely to a continuous linear operator on $\left( \Sigma_\nu^\mu \right)'(\rr d)$, 
and the propagation of singularities follows \eqref{eq:propagationGS1} and \eqref{eq:propagationGS2}.
\end{proof}

Again the overall conclusion is that propagation of singularities works as if $v = 0$.

\subsection{Fourier transformation of the evolution equation}\label{subsec:fouriertransform}

Next we take the Fourier transform $\cF u(t,\cdot)$. 
If we denote this Fourier transform for simplicity still by $u(t,\cdot)$, then 
we obtain from \eqref{eq:airyschrodeq} the evolution equation
\begin{equation}\label{eq:airyschrodeq2}
\left\{
\begin{array}{rl}
\partial_t u(t,x) + i \left( - \la v, D_x \ra  + p(x) \right)u (t,x) & = 0, \quad x \in \rr d, \quad t \in \ro, \\
u(0,\cdot) & = u_0, 
\end{array}
\right.
\end{equation}
where again $v \in \rr d$ is a vector with nonzero entries: $v_j \neq 0$, $1 \leqs j \leqs d$. 

Referring to \eqref{eq:phasefunction} and \eqref{eq:solutionoperator} the solution is now for $u_0 \in \cS(\rr d)$
\begin{equation}\label{eq:solution2}
\begin{aligned}
u(t,x)
& = \wt \cK_t u_0
= \cF  \cK_t \cF^{-1} u_0
& = T_{-t v} \left(  e^{i \fy_t }  u_0 \right) (x) \\
& = e^{i \fy_t (x + t v)}  u_0 (x + t v) \\
& = e^{ - i \fy_{-t} (x)}  u_0 (x + t v). 
\end{aligned}
\end{equation}
The solution operator $\wt \cK_t$ is continuous on $\cS(\rr d)$ and extends to a continuous operator on $\cS'(\rr d)$. 
Now \eqref{eq:propagationG1} combined with \cite[Proposition~4.3~(i)]{Rodino4} give
\begin{equation}\label{eq:propagation3}
\WFgs( \wt \cK_t u_0 )  = \wt \chi_t  \left( \WFgs (u_0) \right), \quad t \in \ro, \quad u_0 \in \cS'(\rr d), \quad \sigma = m-1, 
\end{equation}
where
\begin{equation}\label{eq:hamiltonflow2}
\wt \chi_t (x, \xi)
= \J \chi_t (-\J)(x,\xi)
= (x, \xi - t \nabla P_{m} (x)), \quad t \in \ro, \quad (x, \xi) \in T^* \rr d \setminus 0, 
\end{equation}
and $\chi_t$ is defined by \eqref{eq:hamiltonflow1}. 
This is the Hamilton flow corresponding to the principal part $P_m(x)$ of the polynomial $p(x)$. 
We also obtain
\begin{equation}\label{eq:propagation4}
\WFgs( \wt \cK_t u_0 )  = \WFgs (u_0), \quad t \in \ro, \quad u \in \cS'(\rr d), \quad \sigma > m-1.
\end{equation}

These considerations, combined with a similar discussion in the Gelfand--Shilov framework, may be summarized as follows. 

\begin{thm}\label{thm:propagationAiry3}
Let $p$ be a polynomial with real coefficients defined by \eqref{eq:polynomialsum}, \eqref{eq:polynomialonevar}, 
of order $m = \max_{j=1}^d \deg p_j \geqs 2$, 
with principal part $P_m$ defined by \eqref{eq:principalpart}. 
Denote the Hamilton flow of $P_m(x)$ as in \eqref{eq:hamiltonflow2}. 
Suppose $\wt \cK_t: \cS (\rr d) \to \cS (\rr d)$ is the 
solution operator \eqref{eq:solution2},
where $\fy_t$ is defined by \eqref{eq:primitivepolynomial} and \eqref{eq:phasefunction}, 
for the evolution equation \eqref{eq:airyschrodeq2}. 
Then
\begin{align*}
\WFgs( \wt \cK_t u )  & = \wt \chi_t  \left( \WFgs (u) \right), \quad t \in \ro, \quad u \in \cS'(\rr d), \quad \sigma = m-1, \\
\WFgs( \wt \cK_t u )  & = \WFgs (u), \quad t \in \ro, \quad u \in \cS'(\rr d), \quad \sigma > m-1.
\end{align*}

If $\nu \geqs \mu(m-1) > 1$ then $\wt \cK_t$ is continuous on $\Sigma_\mu^\nu(\rr d)$, extends uniquely to a continuous linear operator on $\left( \Sigma_\mu^\nu \right)'(\rr d)$, and 
\begin{align*}
\WF^{\mu,\nu} ( \wt \cK_t u)  & = \wt \chi_t  \left( \WF^{\mu,\nu} (u) \right), \quad t \in \ro, \quad u \in \left( \Sigma_\mu^\nu \right)'(\rr d), \quad \nu = \mu(m-1) > 1, \\
\WF^{\mu,\nu} ( \wt \cK_t u)  & = \WF^{\mu,\nu} (u), \quad t \in \ro, \quad u \in \left( \Sigma_\mu^\nu \right)'(\rr d), \quad \nu > \mu(m-1) > 1. 
\end{align*}
\end{thm}

The conclusion from Theorem \ref{thm:propagationAiry3} is that the propagation of singularities for 
\eqref{eq:airyschrodeq2} works again exactly as when $v = 0$, 
in both the tempered Schwartz distribution and the Gelfand--Shilov ultradistribution frameworks, respectively.

\begin{rem}\label{rem:homogenhamilton}
Consider the Hamiltonian $a(x,\xi) = p(x) - \la v, \xi \ra$ in the equation \eqref{eq:airyschrodeq2}, with $\deg p = m$. 
The propagation of $\WFgs$ with $\sigma = m-1$ is governed by $a_0(x,\xi) = P_m(x)$
which satisfies 
\begin{equation*}
a_0( \lambda x, \lambda^\sigma \xi) 
= \lambda^{m} P_m(x)
= \lambda^{1+\sigma} a_0(x,\xi), 
\quad (x,\xi) \in T^* \rr d, \quad \lambda > 0.
\end{equation*}
Thus $a_0 \in G^{1 + \sigma,\sigma} = G^{m,m-1}$ according to Corollary \ref{cor:homogeneoussymbol}. 
This is similar to Remark \ref{rem:comparison1}. 
In Section \ref{sec:hamiltonflow} we will study more general Hamiltonians that satisfy this type of anisotropic homogeneity. 
\end{rem}

When the Hamiltonian Weyl symbol $a(x,\xi) = p(\xi)$ is a polynomial in $\xi$ of the form \eqref{eq:polynomialsum} 
then the Hamilton flow is as in \eqref{eq:hamiltonflow1} that is
\begin{equation}\label{eq:hamiltonflow1a}
\chi_t (x, \xi)
= (x + t \nabla P_{m} (\xi), \xi), \quad t \in \ro, \quad (x, \xi) \in T^* \rr d \setminus 0, 
\end{equation}
where $P_m$ is the principal part of $p$. 
When instead the Weyl symbol depends on $x$, $a(x,\xi) = p(x)$, with the same assumptions on $p$, 
we obtain the Hamilton flow \eqref{eq:hamiltonflow2} that is 
\begin{equation}\label{eq:hamiltonflow2a}
\chi_t (x, \xi)
= (x, \xi - t \nabla P_{m} (x)), \quad t \in \ro, \quad (x, \xi) \in T^* \rr d \setminus 0. 
\end{equation}

Define for $\sigma > 0$ the anisotropic scaling map $\Lambda_\sigma (\lambda): T^* \rr d \to T^* \rr d$ as 
\begin{equation}\label{eq:anisotropicscaling}
\Lambda_\sigma (\lambda) (x,\xi) = (\lambda x, \lambda^\sigma \xi), \quad (x,\xi) \in T^* \rr d, \quad \lambda > 0. 
\end{equation}

For suitable $\sigma > 0$ the Hamilton flows \eqref{eq:hamiltonflow1a} and \eqref{eq:hamiltonflow2a} commute with $\Lambda_\sigma (\lambda)$ for all $\lambda > 0$. 
In fact, if $\chi_t$ is defined by \eqref{eq:hamiltonflow1a} and $\sigma = \frac1{m-1}$ then for $\lambda > 0$
\begin{equation*}
\chi_t ( \lambda x, \lambda^\sigma \xi)
= (\lambda x + t \nabla P_{m} (\lambda^\sigma \xi), \lambda^\sigma \xi)
= ( \lambda (x + t \nabla P_{m} (\xi)), \lambda^\sigma \xi ) 
= \Lambda_\sigma (\lambda) \chi_t(x,\xi). 
\end{equation*}
Likewise if $\chi_t$ is defined by \eqref{eq:hamiltonflow2a} and $\sigma = m-1$ then for $\lambda > 0$
\begin{equation*}
\chi_t ( \lambda x, \lambda^\sigma \xi)
= (\lambda x,  \lambda^\sigma \xi - t \nabla P_{m} ( \lambda x))
= ( \lambda x, \lambda^\sigma (\xi - t \nabla P_{m} ( x) ) ) 
= \Lambda_\sigma (\lambda) \chi_t(x,\xi). 
\end{equation*}

Thus in both cases the Hamilton flow $\chi_t$ commutes with anisotropic scaling $\Lambda_\sigma$
\begin{equation}\label{eq:commutativity}
\chi_t \Lambda_\sigma (\lambda) = \Lambda_\sigma (\lambda) \chi_t, \quad \lambda > 0, \quad t \in \ro, 
\end{equation}
suppressing the variables $(x,\xi) \in T^* \rr d \setminus 0$. 

\begin{rem}\label{rem:consistency}
The commutativity \eqref{eq:commutativity} means that the considered Hamilton flows are consistent with the propagation inclusion that we 
aim for, namely
\begin{equation}\label{eq:propagationgeneric1}
\WFgs( \cK_t u)  \subseteq \chi_t  \left( \WFgs (u) \right), \quad t \in \ro, \quad u \in \cS'(\rr d), 
\end{equation}
for the solution operator (propagator) $\cK_t$ of a Schr\"odinger type evolution equation.  

Indeed the inclusion \eqref{eq:propagationgeneric1} requires that the image of $\chi_t$ of any $\WFgs (u) \subseteq T^* \rr d \setminus 0$
for $u \in \cS'(\rr d)$
contains a closed $\sigma$-conic subset of $T^* \rr d \setminus 0$. 
It is not known if for any closed $\sigma$-conic subset of $\Gamma \subseteq T^* \rr d \setminus 0$ there exists $u \in \cS'(\rr d)$ such that $\WFgs (u) = \Gamma$ except when $\sigma = 1$. 
In fact if $\sigma = 1$ then \cite[Theorem~6.1]{Schulz1} answers the question affirmatively. 
Nevertheless it seems reasonable to ask that the image of $\chi_t$ of any closed $\sigma$-conic subset of $T^* \rr d \setminus 0$
contains a closed $\sigma$-conic subset of $T^* \rr d \setminus 0$. 
Then in particular a $\sigma$-conic curve of the form
\begin{equation*}
R_{x,\xi} = \{ (\lambda x, \lambda^\sigma \xi) \in T^* \rr d \setminus 0, \ \lambda > 0 \}
\end{equation*}
where $(x,\xi) \in \sr {2d-1}$, must be mapped into another such curve, that is, 
$\chi_t R_{x,\xi} = R_{z}$ where $z \in \sr {2d-1}$. 
\end{rem}

\section{Anisotropically homogeneous Hamiltonians and their flows}\label{sec:hamiltonflow}

Given a Hamiltonian $a: \rr {2d} \setminus 0 \to \ro$ of class $C^\infty$, Hamilton's system of equations is 
\begin{equation}\label{eq:hamilton}
\left\{
\begin{array}{l}
x' (t) =  \nabla_{\xi} a \left( x(t), \xi(t) \right), \\
\xi' (t) =  -\nabla_{x} a \left( x(t), \xi(t) \right), \\
x(0) =  x, \\
\xi(0) =  \xi,
\end{array}
\right.
\end{equation}
for initial datum $(x,\xi) \in T^* \rr d \setminus 0$ and $t \in (-T,T)$ with $T > 0$. 
By the Picard--Lindel\"of theorem there is a unique solution $( x(t), \xi(t) ) = \chi_t(x,\xi)$, $\chi_t: \rr {2d} \setminus 0 \to \rr {2d} \setminus 0$, $t \in (-T,T)$, for some $T > 0$. It is called the Hamiltonian flow. 
In general the maximal $T$ depends on $(x,\xi)$. 
The map $(-T,T) \ni t \to \chi_t$ satisfies $\chi_{t_1+t_2} = \chi_{t_1} \chi_{t_2}$ and $\chi_t^{-1} = \chi_{-t}$ \cite{Arnold1}. 
The solution $\chi_t$ is a symplectomorphism on $T^* \rr d$ for fixed $t \in (-T,T)$ \cite{deGosson1}, $C^1$ with respect to $t$, and hence a $C^1$ diffeomorphism on $T^* \rr d$. 
If the level sets of $a$ are compact then the solution $\chi_t(x,\xi)$ extends to all $t \in \ro$ \cite{Arnold1}. 
Using the matrix \eqref{eq:Jdef} we may write the differential equation in \eqref{eq:hamilton} as 
\begin{equation}\label{eq:hamiltoncompact}
\begin{aligned}
\left( 
\begin{array}{l}
x' (t)  \\
\xi'(t)
\end{array}
\right)
& = \J \nabla_{x,\xi} a \left( x(t), \xi(t) \right). 
\end{aligned}
\end{equation}

Suppose the solution $\chi_t: \rr {2d} \setminus 0 \to \rr {2d} \setminus 0$ is well defined for $t \in (-T,T)$ with the parameter $T > 0$ 
valid for all initial data $(x,\xi) \in T^* \rr d \setminus 0$. 
The assumption $a \in C^\infty(\rr {2d} \setminus 0)$ and \cite[Theorem~V.4.1]{Hartman1} imply that 
\begin{equation}\label{eq:smoothflow}
(-T,T) \times \rr {2d} \setminus 0 \ni (t, x,\xi) \mapsto \pdd {x} {\alpha} \pdd {\xi} {\beta} \chi_t (x,\xi) \in \rr {2d} \setminus 0 \in C ((-T,T) \times \rr {2d} \setminus 0) \quad \forall \alpha, \beta \in \nn d,
\end{equation}
and in particular $\chi_t \in C^\infty( \rr {2d} \setminus 0, \rr {2d} \setminus 0)$ 
for each $t \in (-T,T)$. 

The next lemma will be used in the proofs of Proposition \ref{prop:homogeneoushamiltonian} and its converse Proposition \ref{prop:scaleinvariancehamiltonian}. 

\begin{lem}\label{lem:gradientanisotropic}
If $\sigma > 0$ and $a \in C^\infty(\rr {2d} \setminus 0)$ is real-valued
then 
\begin{equation}\label{eq:anisohomHamilt0}
a( \lambda x, \lambda^\sigma \xi) = \lambda^{\sigma+1} a(x,\xi), \quad (x,\xi) \in T^* \rr d \setminus 0, \quad \lambda > 0, 
\end{equation}
holds if and only if 
\begin{align}
\lim_{(x,\xi) \to (0,0)} a(x,\xi) & = 0, \label{eq:limitzero} \\
\nabla_x a ( \lambda x, \lambda^\sigma \xi) 
& = \lambda^{\sigma} \nabla_x a ( x, \xi), \quad (x,\xi) \in T^* \rr d \setminus 0, \quad  \lambda > 0, \quad \mbox{and} \label{eq:gradientaniso1} \\
\nabla_\xi a ( \lambda x, \lambda^\sigma \xi) 
& = \lambda \nabla_\xi a ( x, \xi), \quad (x,\xi) \in T^* \rr d \setminus 0, \quad  \lambda > 0, \label{eq:gradientaniso2}
\end{align}
hold. 
\end{lem}

\begin{proof}
It is immediate to see that \eqref{eq:anisohomHamilt0} implies 
\eqref{eq:gradientaniso1} and \eqref{eq:gradientaniso2}. 
Since any $(y,\eta) \in T^* \rr d \setminus 0$ can be written as $(y,\eta) = (\lambda x, \lambda^\sigma \xi)$
for a unique $\lambda > 0$ and a unique $(x,\xi) \in \sr {2d-1}$ \cite[Section~3]{Rodino4}, also 
\eqref{eq:limitzero} follows from \eqref{eq:anisohomHamilt0}. 

Assume on the other hand \eqref{eq:limitzero}, \eqref{eq:gradientaniso1} and \eqref{eq:gradientaniso2}. 
Let $(x,\xi) \in T^* \rr d \setminus 0$ and define the function $f(t) = a (t x, t \xi)$ for $t > 0$. 
Then we have for $0 < \ep < 1$
\begin{align*}
a(x,\xi) = f(1) = \int_{\ep}^1 f'(t) \, \dd t + f (\ep)
= \int_{\ep}^1 \la \nabla_{x,\xi} a(t (x,\xi)), (x,\xi) \ra \, \dd t + a ( \ep( x, \xi ))
\end{align*}
which gives for $\lambda > 0$, using \eqref{eq:gradientaniso1} and \eqref{eq:gradientaniso2}, 
\begin{align*}
a( \lambda x, \lambda^\sigma \xi) 
& = \int_{\ep}^1 \la \nabla_{x,\xi} a(t ( \lambda x, \lambda^\sigma \xi)), ( \lambda x, \lambda^\sigma \xi) \ra \, \dd t + a ( \ep( \lambda x, \lambda^\sigma \xi )) \\
& = \lambda^{\sigma+1} \int_{\ep}^1 \la \nabla_{x,\xi} a(t (x, \xi)), ( x, \xi) \ra \, \dd t + a ( \ep( \lambda x, \lambda^\sigma \xi )) \\
& = \lambda^{\sigma+1} \Big( a(x,\xi) - a ( \ep( x, \xi )) \Big) + a ( \ep( \lambda x, \lambda^\sigma \xi )). 
\end{align*}
The claim \eqref{eq:anisohomHamilt0} now follows from the limit as $\ep \to 0^+$ using the assumption \eqref{eq:limitzero}. 
\end{proof}

In the following result we show that the Hamilton flow commutes with anisotropic scaling for 
Hamiltonians with the anisotropic homogeneity \eqref{eq:anisohomHamilt0}. 

\begin{prop}\label{prop:homogeneoushamiltonian}
Let $\sigma > 0$, 
and suppose $a \in C^\infty(\rr {2d} \setminus 0)$ is real-valued and satisfies
\begin{equation}\label{eq:anisohomHamilt}
a( \lambda x, \lambda^\sigma \xi) = \lambda^{\sigma+1} a(x,\xi), \quad (x,\xi) \in T^* \rr d \setminus 0, \quad \lambda > 0. 
\end{equation}
Then there exists $T > 0$ such that the Hamilton flow $\chi_t(x,\xi)$ defined by the function $a$ 
is well defined for $t \in [-T,T]$ uniformly for all $(x,\xi) \in T^* \rr d \setminus 0$, and $\chi_t$
satisfies
\begin{equation}\label{eq:anisohomHamiltonflow}
\chi_t(  \Lambda_\sigma (\lambda) (x,\xi) ) 
= \Lambda_\sigma (\lambda) \chi_t (x,\xi), \quad \lambda > 0, \quad (x,\xi) \in T^* \rr d \setminus 0, \quad t \in [-T,T],  
\end{equation}
where $\Lambda_\sigma (\lambda): T^* \rr d \to T^* \rr d$ is defined in \eqref{eq:anisotropicscaling}. 
\end{prop}

\begin{proof}
The assumption \eqref{eq:anisohomHamilt} and Lemma \ref{lem:gradientanisotropic} give the anisotropic homogeneities
\begin{equation*}
\begin{aligned}
\nabla_x a ( \lambda x, \lambda^\sigma \xi) 
& = \lambda^{\sigma} \nabla_x a ( x, \xi), \\
\nabla_\xi a ( \lambda x, \lambda^\sigma \xi) 
& = \lambda \nabla_\xi a ( x, \xi), \quad (x,\xi) \in T^* \rr d \setminus 0, \quad  \lambda > 0, 
\end{aligned}
\end{equation*}
which can be written as
\begin{equation}\label{eq:anisogradient}
\nabla_{x,\xi} a ( \lambda x, \lambda^\sigma \xi) 
 = \Lambda_{\frac1\sigma} (\lambda^{\sigma}) \nabla_{x,\xi} a ( x, \xi). 
\end{equation}

This gives $\lim_{(x,\xi) \to (0,0)} \nabla_{x,\xi} a(x,\xi) = 0$. 
Set
\begin{equation*}
M = \sup_{0 < |(x,\xi)| \leqs \frac32} | \nabla_{x,\xi} a(x,\xi)| < + \infty. 
\end{equation*}
If $(x,\xi) \in \sr {2d-1}$ then by the Picard--Lindel\"of theorem \cite[Theorem~II.1.1]{Hartman1}
the Hamilton flow stays in the ball $\chi_t(x,\xi) \in \overline{\rB_{\frac12}(x,\xi)}$ if $-T \leqs t \leqs T$
and $T = \frac1{2M}$. 
Thus there exists $T > 0$ such that the Hamilton flow $\chi_t: \sr {2d-1} \to \rr {2d} \setminus 0$ is well defined for
$-T \leqs t \leqs T$ uniformly over $\sr {2d-1}$. 

Let $(x,\xi) \in T^* \rr d \setminus 0$ and set $(x(t), \xi(t) ) = \chi_t(x,\xi)$.
For $T_0 > 0$ sufficiently small we have $( x(t), \xi(t) ) \in T^* \rr d \setminus 0$ for $t \in [-T_0,T_0]$. 
From \eqref{eq:hamiltoncompact}, \eqref{eq:anisogradient} and 
\begin{equation}\label{eq:Jcommutator}
\J \Lambda_{\frac1\sigma} (\lambda^{-\sigma}) = \Lambda_{\sigma} (\lambda^{-1}) \J
\end{equation}
we obtain
\begin{align*}
\frac{d}{dt} \chi_t(x,\xi) 
= 
\left( 
\begin{array}{l}
x' (t)  \\
\xi'(t)
\end{array}
\right)
& = \J \nabla_{x,\xi} a \left( x(t), \xi(t) \right) \\
& =  \J \Lambda_{\frac1\sigma} (\lambda^{-\sigma}) \nabla_{x,\xi} a \left( \lambda x(t), \lambda^\sigma \xi(t) \right) \\
& =  \Lambda_{\sigma} (\lambda^{-1}) \J \nabla_{x,\xi} a \left( \lambda x(t), \lambda^\sigma \xi(t) \right)
\end{align*}
which may be written
\begin{equation*}
( \lambda x'(t), \lambda^\sigma \xi'(t) ) = \J \nabla_{x,\xi} a \left( \lambda x(t), \lambda^\sigma \xi(t) \right).
\end{equation*}

Thus $( \lambda x(t), \lambda^\sigma \xi(t) )$ solves \eqref{eq:hamilton} for $t \in [-T_0,T_0]$ with initial datum $(\lambda x, \lambda^\sigma \xi)$, 
for any $\lambda > 0$. If we choose $\lambda > 0$ such that $|(\lambda x, \lambda^\sigma \xi)| = 1$ then the solution is well defined for 
$t \in [-T,T]$ by the first part of the proof.
The solution $( \lambda x(t), \lambda^\sigma \xi(t) )$ hence extends to $t \in [-T,T]$ for all $\lambda > 0$. 
By the uniqueness of the solution we have $\chi_t (\lambda x, \lambda^\sigma \xi) = ( \lambda x(t), \lambda^\sigma \xi(t) )$. 
It follows that the Hamilton flow $\chi_t: \rr {2d} \setminus 0 \to \rr {2d} \setminus 0$ is well defined in the interval $t \in [-T,T]$
uniformly over the phase space $\rr {2d} \setminus 0$. 
In conclusion we have 
\begin{equation*}
\Lambda_\sigma (\lambda) \chi_t(x,\xi) = ( \lambda x(t), \lambda^\sigma \xi(t) ) = \chi_t( \lambda x, \lambda^\sigma \xi) = \chi_t ( \Lambda_\sigma(\lambda) (x,\xi) )
\end{equation*}
for $(x,\xi) \in T^* \rr d \setminus 0$, $\lambda > 0$ and $t \in [-T,T]$. 
\end{proof}

\begin{rem}\label{rem:limitzero}
With the assumptions of Proposition \ref{prop:homogeneoushamiltonian}, 
for any $t \in [-T,T]$ we have 
\begin{equation*}
\lim_{(x,\xi) \to (0,0)} \chi_t (x,\xi)  = 0.  
\end{equation*}
In fact this is an immediate consequence of \eqref{eq:anisohomHamiltonflow}. 
So defining $\chi_t(0,0) = (0,0)$ we could extend the Hamilton flow 
as a continuous bijection $\chi_t: \rr {2d} \to \rr {2d}$ for $t \in [-T,T]$. 
By \cite[Theorem~V.4.1]{Hartman1} we know that 
$\chi_t \in C^\infty( \rr {2d} \setminus 0, \rr {2d} \setminus 0)$
but we cannot extend the smoothness to the new domain point $(0,0)$. 
\end{rem}

Next we show a converse of Proposition \ref{prop:homogeneoushamiltonian}. 

\begin{prop}\label{prop:scaleinvariancehamiltonian}
Let $a \in C^\infty(\rr {2d} \setminus 0)$ be real-valued and suppose $\lim_{(x,\xi) \to (0,0)} a(x,\xi) = 0$. 
Suppose
the solution $\chi_t(x,\xi)$ to \eqref{eq:hamilton} is well defined for $t \in [-T,T]$ for some $T > 0$
for all $(x,\xi) \in T^* \rr d \setminus 0$. 
If  $\sigma > 0$ and 
\eqref{eq:anisohomHamiltonflow} holds true then $a$ satisfies the homogeneity
\begin{equation}\label{eq:shomogenous1}
a( \lambda x, \lambda^\sigma \xi) = \lambda^{\sigma+1} a( x, \xi), \quad (x,\xi) \in T^* \rr d \setminus 0, \quad \lambda > 0. 
\end{equation}
\end{prop}

\begin{proof}
For $(x,\xi) \in T^* \rr d \setminus 0$ we denote $( x(t), \xi(t) ) = \chi_t(x,\xi)$.  
Formula \eqref{eq:anisohomHamiltonflow} means that the solution to \eqref{eq:hamilton} with $(x,\xi)$
replaced by $(\lambda x, \lambda^\sigma \xi)$ for $\lambda > 0$ is $\Lambda_\sigma (\lambda) \chi_t( x, \xi ) = (\lambda x(t), \lambda^\sigma \xi(t) )$. 

Let $(x,\xi) \in T^* \rr d \setminus 0$.
From \eqref{eq:hamiltoncompact} and \eqref{eq:anisohomHamiltonflow} we obtain for any $\lambda > 0$
\begin{align*}
\J \nabla_{x,\xi} a \left( x(t), \xi(t) \right)
& = \frac{d}{dt} \chi_t(x,\xi) 
= \frac{d}{dt} \left( \Lambda_\sigma (\lambda^{-1})  \chi_t \left( \Lambda_\sigma (\lambda) (x,\xi) \right) \right) \\
& = \Lambda_\sigma (\lambda^{-1}) \frac{d}{dt} \left( \chi_t \left( \Lambda_\sigma (\lambda) (x,\xi) \right) \right) \\
& = \Lambda_\sigma (\lambda^{-1}) \J \nabla_{x,\xi} a \left( \lambda x(t), \lambda^\sigma \xi(t) \right). 
\end{align*}
With aid of \eqref{eq:Jcommutator} and $\J^{-1} = -\J$ this gives 
\begin{align*}
\nabla_{x,\xi} a \left( x(t), \xi(t) \right)
& = - \J \Lambda_\sigma (\lambda^{-1}) \J \nabla_{x,\xi} a \left( \lambda x(t), \lambda^\sigma \xi(t) \right) \\
& = \Lambda_{\frac1\sigma} (\lambda^{-\sigma}) \nabla_{x,\xi} a \left( \lambda x(t), \lambda^\sigma \xi(t) \right). 
\end{align*}
For $t=0$ we get 
\begin{align*}
\nabla_x a ( \lambda x, \lambda^\sigma \xi) 
& = \lambda^\sigma \nabla_x a ( x, \xi), \\
\nabla_\xi a ( \lambda x, \lambda^\sigma \xi) 
& = \lambda \nabla_\xi a ( x, \xi) 
\end{align*}
which together with the assumption $\lim_{(x,\xi) \to (0,0)} a(x,\xi) = 0$
is equivalent to \eqref{eq:shomogenous1}
by Lemma \ref{lem:gradientanisotropic}. 
\end{proof}

We note that a function $a$ that satisfies \eqref{eq:shomogenous1} is determined by its values on the unit sphere $\sr {2d-1}$, 
and
\begin{equation*}
a( x, \xi) = \lambda_\sigma^{\sigma+1} (x,\xi) \, a( p_\sigma(x,\xi) ), \quad (x,\xi) \in T^* \rr d \setminus 0, 
\end{equation*}
where $\lambda_\sigma: \rr {2d} \setminus 0 \to \ro_+$ and $p_\sigma: \rr {2d} \setminus 0 \to \sr {2d-1}$ are smooth functions defined in \cite[Section~3]{Rodino4}. 

Examples of Hamiltonians that satisfy $a \in C^\infty(\rr {2d} \setminus 0)$ and \eqref{eq:shomogenous1} are 
the homogeneous polynomials that depend on either $x$ or $\xi$ (but not both) studied in Section \ref{sec:airyschrodinger}
(cf. Remarks \ref{rem:comparison1} and \ref{rem:homogenhamilton}), that is 
\begin{equation*}
a(x,\xi) = P_m(x), \quad \sigma = m-1, 
\end{equation*}
\begin{equation*}
a(x,\xi) = P_m(\xi), \quad \sigma = \frac{1}{m-1}, 
\end{equation*}
where $m \in \no$ and $m \geqs 2$.
Other examples are
\begin{equation*}
a(x,\xi) = c \left( |x| + |\xi|^{\frac1\sigma} \right)^{\sigma+1}, 
\end{equation*}
where $\sigma > 0$ and $c \in \ro \setminus 0$, and 
\begin{equation*}
a(x,\xi) = c \left( |x|^{2k} + |\xi|^{2m} \right)^{ \frac12 \left( \frac1k + \frac1m \right)}, 
\end{equation*}
with $k,m \in \no \setminus 0$, $\sigma = \frac{k}{m}$ and $c \in \ro \setminus 0$. 

Note that the Hamiltonians 
\begin{equation*}
a(x,\xi) = c_1 |x|^{\sigma+1} + c_2 |\xi|^{1 + \frac1\sigma} 
\end{equation*}
with $c_1, c_2 \in \ro$ and $\sigma > 0$, 
\begin{equation*}
a(x,\xi) = c_1 |x|^{2k} + c_2 |\xi|^{1 + \frac{1}{2k-1}}
\end{equation*}
with $k \in \no \setminus 0$, $\sigma = 2k-1$ and $c_1, c_2 \in \ro$,
and
\begin{equation*}
a(x,\xi) = c_1 |x|^{1 + \frac{1}{2k-1}} + c_2 |\xi|^{2k} 
\end{equation*}
with $\sigma = \frac1{2k-1}$ and $c_1, c_2 \in \ro$, 
all satisfy \eqref{eq:shomogenous1}. 
But none of them satisfy $a \in C^\infty(\rr {2d} \setminus 0)$ unless $\sigma = 1$ and $k = 1$ respectively. 

The final result in this section will be useful in Section \ref{sec:propagation}. 
It says that the $G^{m,\sigma}$ property of a symbol is preserved under composition with 
a Hamiltonian flow that satisfies the anisotropic scaling commutativity \eqref{eq:anisohomHamiltonflow}. 
We need a cutoff function $\psi_\delta(x,\xi) = \fy(|x|^2 + |\xi|^2) \in C^\infty(\rr {2d})$ where $\fy \in C^{\infty}(\ro)$,
$0 \leqs \fy \leqs 1$, $\fy(t) = 0$ for $t \leqs \frac{\delta^2}{4}$ and $\fy(t) = 1$ for $t \geqs \delta^2$
for a given $\delta > 0$. 
Thus $\psi_\delta \big|_{\rB_{\frac{\delta}{2}}} \equiv 0$ and $\psi_\delta \big|_{\rr {2d} \setminus \rB_\delta} \equiv 1$. 

\begin{prop}\label{prop:compsymbolflow}
Let $\sigma, \delta, T > 0$, and suppose 
$\chi_t \in C^\infty( \rr {2d} \setminus 0, \rr {2d} \setminus 0)$ for $-T \leqs t \leqs T$
is a Hamiltonian flow that satisfies the anisotropic scaling commutativity \eqref{eq:anisohomHamiltonflow}.
If $a \in G^{m,\sigma}$ then $b_t = \psi_\delta  (a \circ \chi_t) \in G^{m,\sigma}$ 
uniformly for all $-T \leqs t \leqs T$.
\end{prop}

\begin{proof}
Let $(x,\xi) \in T^* \rr d$ satisfy $|(x,\xi)| \geqs \delta$ and let $\lambda \geqs 1$. 
From \eqref{eq:anisohomHamiltonflow} we obtain
\begin{equation*}
b_t( \lambda x, \lambda^\sigma \xi) 
= a \left( \chi_t ( \lambda x, \lambda^\sigma \xi ) \right)
= a \left( \Lambda_\sigma (\lambda) \chi_t ( x, \xi ) \right)
= a \left( \lambda \chi_{t,1} ( x, \xi ) , \lambda^\sigma \chi_{t,2} ( x, \xi ) \right)
\end{equation*}
decomposing $\chi_t = ( \chi_{t,1}, \chi_{t,2} )$ into its two $\rr d$ component functions. 
For $1 \leqs k \leqs d$ we denote by $\chi_{t,j,k}$ the component with index $k$ of $\chi_{t,j}$ for $j = 1,2$. 

We claim that for $|(x,\xi)| > \delta$, $\lambda \geqs 1$, and $\alpha, \beta \in \nn d$ we have 
\begin{equation}\label{eq:formulacomposition}
\begin{aligned}
\pdd x \alpha \pdd \xi \beta \Big( a \left( \lambda \chi_{t,1} ( x, \xi ) , \lambda^\sigma \chi_{t,2} ( x, \xi )\right) \Big)
= 
\sum_{|\gamma + \kappa| \leqs |\alpha+\beta|}
\lambda^{|\gamma| + \sigma |\kappa|} 
\left( \pdd x \gamma \pdd \xi \kappa a \right) \Big( \Lambda_\sigma (\lambda) \chi_t ( x, \xi ) \Big)
f_{\gamma,\kappa}(x,\xi)
\end{aligned}
\end{equation}
where $f_{\gamma,\kappa} \in C^\infty( \rr {2d} \setminus 0)$ are smooth functions. 
In fact the claim follows by induction with respect to $|\alpha+\beta|$, starting with $|\alpha+\beta| = 1$, as follows. 
With $e_k \in \nn d$ denoting the standard basis vector, $1 \leqs k \leqs d$, we may write
$\partial_{x_k} a(x,\xi) = \pdd x {e_k} a(x,\xi)$ and $\partial_{\xi_k} a(x,\xi) = \pdd \xi {e_k} a(x,\xi)$. 
If $|\alpha+\beta| = 1$ we have either
\begin{align*}
& \partial_{x_j} \Big( a \left( \lambda \chi_{t,1} ( x, \xi ) , \lambda^\sigma \chi_{t,2} ( x, \xi )\right) \Big) \\
& = 
\sum_{k=1}^d
\left( 
\lambda \left( \pdd x {e_k} a \right) \Big( \Lambda_\sigma (\lambda) \chi_t ( x, \xi ) \Big) 
\frac{\partial \chi_{t,1,k}}{\partial x_j} (x,\xi)
+ \lambda^\sigma \left( \pdd \xi {e_k} a \right) \Big( \Lambda_\sigma (\lambda) \chi_t ( x, \xi ) \Big) 
\frac{\partial \chi_{t,2,k}}{\partial x_j} (x,\xi)
\right)
\end{align*}
or 
\begin{align*}
& \partial_{\xi_j} \Big( a \left( \lambda \chi_{t,1} ( x, \xi ) , \lambda^\sigma \chi_{t,2} ( x, \xi )\right) \Big) \\
& = 
\sum_{k=1}^d
\left( 
\lambda \left( \pdd x {e_k} a \right) \Big( \Lambda_\sigma (\lambda) \chi_t ( x, \xi ) \Big) 
\frac{\partial \chi_{t,1,k}}{\partial \xi_j} (x,\xi)
+ \lambda^\sigma \left( \pdd \xi {e_k} a \right) \Big( \Lambda_\sigma (\lambda) \chi_t ( x, \xi ) \Big) 
\frac{\partial \chi_{t,2,k}}{\partial \xi_j} (x,\xi)
\right) 
\end{align*}
for $1 \leqs j \leqs d$. Thus \eqref{eq:formulacomposition} holds when $|\alpha+\beta| = 1$.
The induction step follows straight-forwardly. 
It follows that \eqref{eq:formulacomposition} holds for all $\alpha, \beta \in \nn d$, $|(x,\xi)| > \delta$, and $\lambda \geqs 1$, 
as claimed.

We fix $r > \delta$ and consider any $(x,\xi) \in T^* \rr d$ such that $|(x,\xi)| = r$. 
Using \eqref{eq:formulacomposition}, the assumption $a \in G^{m,\sigma}$ and 
\begin{equation*}
\inf_{\substack{ |t| \leqs T \\ |(x,\xi)| = r}} |\chi_t (x,\xi)| > 0, \quad \sup_{\substack{ |t| \leqs T \\ |(x,\xi)| = r}} |\chi_t (x,\xi)| < \infty, 
\end{equation*}
we estimate for $\alpha, \beta \in \nn d$
\begin{align*}
\lambda^{|\alpha| + \sigma |\beta|} \left| \Big( \pdd x \alpha \pdd \xi \beta b_t \Big) ( \lambda x, \lambda^\sigma \xi) \right|
& = \left| \pdd x \alpha \pdd \xi \beta \Big( b_t ( \lambda x, \lambda^\sigma \xi) \Big) \right| \\
& = \left| \pdd x \alpha \pdd \xi \beta \Big( a \left( \Lambda_\sigma (\lambda) \chi_t ( x, \xi ) \right) \Big) \right| \\
& \lesssim \sum_{|\gamma + \kappa| \leqs |\alpha+\beta|}
\lambda^{|\gamma| + \sigma |\kappa|} 
\left| \left( \pdd x \gamma \pdd \xi \kappa a \right) \Big( \Lambda_\sigma (\lambda) \chi_t ( x, \xi ) \Big) \right| \\
& \lesssim \sum_{|\gamma + \kappa| \leqs |\alpha+\beta|}
\lambda^{|\gamma| + \sigma |\kappa|} 
\theta_\sigma \left( \lambda \chi_{t,1} ( x, \xi ) , \lambda^\sigma \chi_{t,2} ( x, \xi ) \right)^{m - |\gamma| - \sigma |\kappa|} \\
& \lesssim \lambda^m
\end{align*}
for all $\lambda \geqs 1$. 
The conclusion $b_t \in G^{m,\sigma}$ uniformly for all  $-T \leqs t \leqs T$ is now a consequence of Lemma \ref{lem:symbol}. 
\end{proof}

\section{Solutions to a class of Schr\"odinger type equations with anisotropic Hamiltonians}\label{sec:solutionsaniso}

In the sequel we use $k,m \in \no \setminus 0$ and $\sigma = \frac{k}{m}$. 
We consider in this section first the Cauchy problem
\begin{equation}\label{eq:anisoschrodeq}
\left\{
\begin{array}{rl}
\partial_t u(t,x) + i a^w(x,D) u (t,x) & = f(t,x), \quad x \in \rr d, \quad 0 < t \leqs T, \\
u(0,\cdot) & = u_0,
\end{array}
\right.
\end{equation}
where $T > 0$ and $a \in G^{1 + \sigma,\sigma}$.
Later we will extend the time domain to $[-T,T]$. 
 
Simplifying notation we set $M_s = M_{\sigma,s}(\rr d)$ for $s \in \ro$ and $a^w = a^w(x,D)$. 
The main purpose of the section is to show existence and uniqueness of solutions to \eqref{eq:anisoschrodeq}
considering $u(t,\cdot)$ as a function of $t$ with values in $M_s$ spaces. 

We will need the following lemma
which says that $C^1([0,T], \cS)$ is dense in $C([0,T], M_\mu) \cap C^1 ([0,T], M_\nu)$ for any $\mu, \nu \in \ro$. 

\begin{lem}\label{lem:approximation}
If $\mu, \nu \in \ro$ and $u \in C([0,T], M_\mu) \cap C^1 ([0,T], M_\nu)$
then there exists a sequence $(u_n)_{n \geqs 1} \subseteq C^1([0,T], \cS)$
such that 
\begin{align}
& \lim_{n \to +\infty} \sup_{0 \leqs t \leqs T} \| u_n(t,\cdot) - u(t,\cdot) \|_{M_\mu} = 0, \label{eq:claimapprox1} \\
& \lim_{n \to +\infty} \sup_{0 \leqs t \leqs T} \| \partial_t u_n(t,\cdot) - \partial_t u(t,\cdot) \|_{M_\nu} = 0. \label{eq:claimapprox2}
\end{align}
\end{lem}

\begin{proof}
Let $\fy \in \cS(\rr d)$ satisfy $\| \fy \|_{L^2} = 1$. 
We use the approximations (cf. \cite{Grochenig1})
\begin{equation*}
u_n (t,\cdot) = V_\fy^* \chi_n V_\fy u(t,\cdot) \in \cS(\rr d), \quad 0 \leqs t \leqs T,
\end{equation*}
where $\chi_n$ is the indicator function of the ball $\rB_n \subseteq \rr {2d}$, $n \in \no \setminus 0$. 

By \cite[Eq.~(11.29)]{Grochenig1} we have on the one hand
\begin{equation}\label{eq:STFTapprox1}
\left| V_\fy \left( u_n (t,\cdot) - u_n (\tau,\cdot)\right) \right|
\leqs \left( \chi_n \left| V_\fy \left( u (t,\cdot) - u (\tau,\cdot) \right) \right| \right) * |V_\fy \fy|
\end{equation}
and on the other hand,
using \eqref{eq:moyal}
in the form $V_\fy^* V_\fy = \rm{id}_{\cS'}$, 
\begin{equation}\label{eq:STFTapprox2}
\left| V_\fy \left( u_n (t,\cdot) - u (t,\cdot)\right) \right|
\leqs \left( ( 1 - \chi_n ) \left| V_\fy u (t,\cdot) \right| \right) * |V_\fy \fy|. 
\end{equation}

With $m \geqs 0$ we write using \eqref{eq:peetreanisotropic} and \eqref{eq:sobolevweightestimate1}
\begin{align*}
\eabs{z}^m 
& \lesssim \theta_\sigma (z)^{m \max(1,\sigma)}
\lesssim \theta_\sigma (z-w)^{m \max(1,\sigma)} \theta_\sigma (w)^{m \max(1,\sigma)} \\
& \leqs \theta_\sigma (z-w)^{m \max(1,\sigma) + |\mu| + \mu}  \theta_\sigma (w)^{m \max(1,\sigma)}, \quad z,w \in \rr {2d}, 
\end{align*}
which inserted into \eqref{eq:STFTapprox1} gives by means of the Cauchy--Schwarz inequality, again \eqref{eq:sobolevweightestimate1} and \eqref{eq:STFTschwartz}
\begin{align*}
&\eabs{z}^m  \left| V_\fy \left( u_n (t,\cdot) - u_n (\tau,\cdot)\right) (z) \right| \\
& \lesssim \left( \chi_n \theta_\sigma^{m \max(1,\sigma) +|\mu|} \theta_\sigma^\mu \left| V_\fy \left( u (t,\cdot) - u (\tau,\cdot) \right) \right| \right) * \left( \theta_\sigma^{m \max(1,\sigma)} |V_\fy \fy| \right)(z) \\
& \leqs \sup_{\rr {2d}} \left( \chi_n \theta_\sigma^{m\max(1,\sigma) + |\mu|} \right) \left\| \theta_\sigma^\mu \left| V_\fy \left( u (t,\cdot) - u (\tau,\cdot) \right) \right| \right\|_{L^2(\rr {2d})}
\left\|  \theta_\sigma^{m \max(1,\sigma)}  |V_\fy \fy| \right\|_{L^2(\rr {2d})} \\
& \lesssim \| u (t,\cdot) - u (\tau,\cdot) \|_{M_\mu}. 
\end{align*}

Referring to the assumption $u \in C([0,T], M_\mu)$ and 
to the seminorms \eqref{eq:seminormsS} this shows that $u_n \in C([0,T], \cS)$, and $u_n \in C^1([0,T], \cS)$ follows similarly
from $\partial_t u_n (t,\cdot) = V_\fy^* \chi_n V_\fy \partial_t u(t,\cdot)$, replacing $\mu$ with $\nu$
and using the assmption $u \in C^1 ([0,T], M_\nu)$. 

From \eqref{eq:STFTapprox2} and Young's inequality we obtain, again using \eqref{eq:peetreanisotropic}, \eqref{eq:STFTschwartz} and \eqref{eq:sobolevweightestimate1}, 
\begin{align*}
\| u_n(t,\cdot) - u(t,\cdot) \|_{M_\mu}
& = \left\| \theta_\sigma^\mu \left| V_\fy \left( u_n (t,\cdot) - u (t,\cdot) \right) \right| \right\|_{L^2(\rr {2d})} \\
& \lesssim \left\|  \left( ( 1 - \chi_n ) \theta_\sigma^\mu \left| V_\fy u (t,\cdot) \right| \right) * \left( \theta_\sigma^{|\mu|} |V_\fy \fy| \right) \right\|_{L^2(\rr {2d})} \\
& \lesssim \left\|  ( 1 - \chi_n ) \theta_\sigma^\mu \, V_\fy u (t,\cdot) \right\|_{L^2(\rr {2d})}  \left\|  \theta_\sigma^{|\mu|} V_\fy \fy \right\|_{L^1(\rr {2d})} \\
& \lesssim \left\|  ( 1 - \chi_n ) \theta_\sigma^\mu \, V_\fy u (t,\cdot) \right\|_{L^2(\rr {2d})} := f_n (t). 
\end{align*}

Note the monotonicity $f_n (t) \geqs f_{n+1} (t)$ for each $n \in \no \setminus 0$, and 
by the assumption $u \in C([0,T], M_\mu)$ and dominated convergence we get $\lim_{n \to \infty} f_n (t) = 0$
for each $t \in [0,T]$. 
For each $n \in \no \setminus 0$ we have $f_n \in C([0,T])$. In fact
\begin{align*}
| f_n (t) - f_n (\tau)| 
& = \left| \left\| ( 1 - \chi_n ) \theta_\sigma^\mu V_\fy u (t,\cdot) \right\|_{{L^2(\rr {2d})}} - \left\| (1 - \chi_n ) \theta_\sigma^\mu V_\fy u (\tau,\cdot) \right\|_{{L^2(\rr {2d})}} \right| \\
& \leqs \left\| ( 1 - \chi_n ) \theta_\sigma^\mu \left( V_\fy u (t,\cdot)  - V_\fy u (\tau,\cdot)  \right)\right\|_{{L^2(\rr {2d})}} \\
& \leqs \left\| \theta_\sigma^\mu \left( V_\fy u (t,\cdot)  - V_\fy u (\tau,\cdot)  \right)\right\|_{{L^2(\rr {2d})}} \\
& = \| u (t,\cdot) - u (\tau,\cdot) \|_{M_\mu} 
\end{align*}
so $f_n \in C([0,T])$ follows from the assumption $u \in C([0,T], M_\mu)$. 
Now it follows from Dini's theorem that $f_n (t) \to 0$ uniformly for $t \in [0,T]$ as $n \to \infty$. 
This means that we have shown \eqref{eq:claimapprox1}, and \eqref{eq:claimapprox2} follows in the same fashion. 
\end{proof}

\begin{rem}\label{rem:neginterval1}
We note that 
Lemma \ref{lem:approximation} is true also when we 
replace the interval $[0,T]$ with $[-T,T]$. 
\end{rem}

By \eqref{eq:weightequivalence} we have $w_{k,m}^{1/k} \asymp \theta_\sigma$
when $\sigma = \frac{k}{m}$ and $k,m \in \no \setminus 0$. 
Combining this with \eqref{eq:SSmodnorm}, \eqref{eq:locopmodspace} and \cite[Theorem~1.1]{Grochenig2}
it follows that if $s \in \ro$ then
the symbol $\theta_\sigma^s$ for the localization operator \eqref{eq:locopmodspace}
that defines the isometry $M_s \to L^2$ can be replaced by $w_{k,m}^{s/k}$. 
We denote for simplicity this localization operator by $A_s = A_{w_{k,m}^{s/k}}$.   
We will need the following auxiliary result. 

\begin{lem}\label{lem:locopcalculus}
Let $k,m \in \no \setminus 0$, $\sigma = \frac{k}{m}$ and
$a \in G^{1 + \sigma,\sigma}$, and suppose that
\begin{equation*}
\im \, a (x,\xi) \leqs C_1, \quad (x,\xi) \in T^* \rr d, 
\end{equation*}
for $C_1 > 0$. 
If $s \in \ro$ then $A_s \, a^w A_s^{-1} = b^w$
where $b \in G^{1 + \sigma,\sigma}$ and 
\begin{equation*}
\im \, b (x,\xi) \leqs C_2, \quad (x,\xi) \in T^* \rr d, 
\end{equation*}
for some $C_2 > 0$. 
\end{lem}

\begin{proof}
By Lemma \ref{lem:sobolevsymbol} we have $w_{k,m}^{s/k} \in G^{s,\sigma}$ for the symbol of $A_s$. 
From \cite[Theorem~1.7.10]{Nicola1} it follows that $A_s = a_1^w$ where $a_1 \in G^{s,\sigma}$
is real-valued, cf. Section \ref{sec:ShubinSobolev}. 

The symbol $w_{k,m}^{s/k}$ for $A_s$ is positive everywhere and elliptic, cf. \eqref{eq:ellipticity}. 
By the proof of \cite[Theorem~8.2]{Boggiatto1} (cf. also \cite[Proposition~1.7.12]{Nicola1}), slightly modified to the anisotropic calculus, 
it follows that $A_s$ is invertible on $\cS$, and $A_s^{-1} = c^w$
where $c \in G^{-s,\sigma}$. 
From 
\begin{equation*}
(A_s f, f) = \int_{\rr {2d}} w_{k,m}^{s/k} (z) |V_\fy f (z)|^2 \, \dd z > 0
\end{equation*}
for all $f \in \cS \setminus 0$ it follows that $(c^w f, f) > 0$ for all $f \in \cS \setminus 0$ which implies that $c$ is a real-valued symbol. 
Indeed we have 
\begin{align*}
2i \, ( (\im \, c)^w f,f)
= ( c^w f,f) - (\overline c^w f, f)
= ( c^w f,f) - (f, c^w f)
= ( c^w f,f) - \overline{(c^w f,f) }
= 0
\end{align*}
for all $f \in \cS$, which by polarization yields
\begin{align*}
4 ( (\im \, c)^w f,g )
& = \left( (\im \, c)^w (f+g),f+g \right) - \left( (\im \, c)^w (f-g),f-g \right) \\
& + i \left( (\im \, c)^w (f+ig),f+ig \right) - i \left( (\im \, c)^w (f-ig),f-ig \right)
= 0
\end{align*}
for all $f,g \in \cS$. 
This implies $\im \, c \equiv 0$. 

Finally from $b^w = A_s \, a^w A_s^{-1} = a_1^w a^w c^w$
and \eqref{eq:calculuscomposition1}
we obtain 
\begin{equation*}
b = a_1 \wpr a \wpr c = a_1 \, a \, c + b_1
\end{equation*}
where $b_1 \in G^{0, \sigma}$ is bounded. 
Thus since $a_1 \, c \in G^{0, \sigma}$ is also bounded we get
\begin{equation*}
\im \, b = (\im \, a) \,  a_1 \, c + \im \, b_1 \leqs C_1 \sup_{\rr {2d}} \left( a_1 \, c \right) + \sup_{\rr {2d}} \im \, b_1 
= C_2 < \infty
\end{equation*}
for some $C_2 > 0$. 
\end{proof}

\begin{rem}\label{rem:abssymbolbound}
The proof of Lemma \ref{lem:locopcalculus} shows that from the added assumption
\begin{equation*}
|\im \, a (x,\xi)| \leqs C_1, \quad (x,\xi) \in T^* \rr d, 
\end{equation*}
follows the stronger conclusion
\begin{equation*}
|\im \, b (x,\xi) | \leqs C_2, \quad (x,\xi) \in T^* \rr d. 
\end{equation*}
\end{rem}

The following two results Lemma \ref{lem:wellposedness} 
and Theorem \ref{thm:uniquesolutionMsigma} are detailed adaptations of \cite[Lemma~23.1.1 and Theorem~23.1.2]{Hormander1} 
from the calculus of H\"ormander symbols to the anisotropic Shubin calculus. 

\begin{lem}\label{lem:wellposedness}
Let $k,m \in \no \setminus 0$, $\sigma = \frac{k}{m}$, $a \in G^{1 + \sigma, \sigma}$, and suppose that
\begin{equation}\label{eq:imsymbolupperbound}
\im \, a (x,\xi) \leqs C, \quad (x,\xi) \in T^* \rr d, 
\end{equation}
for $C > 0$. 
If $s \in \ro$, $u \in  C([0,T], M_{s + 1 + \sigma} ) \cap C^1 ([0,T], M_s)$ then
\begin{equation}\label{eq:fregular}
f (t) = \partial_t u (t,\cdot) + i a^w u(t,\cdot) \in C([0,T], M_s ), 
\end{equation}
and there exists $c > 0$ such that
\begin{equation}\label{eq:wellposedinequality}
\| u(t) \|_{M_s} \lesssim e^{c t} \| u(0) \|_{M_s} + \int_0^t e^{c (t-\tau)} \| f(\tau) \|_{M_s} \dd \tau
\end{equation}
for $0 \leqs t \leqs T$. 
\end{lem}

\begin{proof}
First we prove the result for $s = 0$. The assumptions, $M_{1 + \sigma} \subseteq L^2$ and Proposition \ref{prop:continuitySobolev} imply
\begin{equation*}
\left\{
\begin{array}{l}
t \mapsto u (t,\cdot) 
\in C([0,T], L^2 ), \\
t \mapsto \partial_t u (t,\cdot) 
\in C([0,T], L^2 ), \\
t \mapsto a^w u (t,\cdot) 
\in C([0,T], L^2 )
\end{array}
\right.
\Longrightarrow \quad t \mapsto f (t) 
\in C([0,T], L^2 ). 
\end{equation*}
The conclusion \eqref{eq:fregular} follows. 

By Lemma \ref{lem:approximation} we may replace $L^2$ by $\cS$ above. 
The assumptions $a \in G^{1 + \sigma, \sigma}$ and $\im \, a (x,\xi) \leqs C$ make
Lemma \ref{lem:Garding} applicable. 
Combining with \eqref{eq:wignerweyl} and the fact that $W(g,g)$ is real-valued \cite{Grochenig1} 
we get for $g \in \cS(\rr d)$
\begin{align*}
\re ( i a^w g, g)
& = - \im ( a^w g, g)
= - (2 \pi)^{-d} \im ( a, W(g,g) )
= - (2 \pi)^{-d} (  \im \, a, W(g,g) ) \\
& = - ( (\im \, a)^w g, g)
= ( ( C - \im \, a)^w g,g) - C \| g \|_{L^2}^2
\geqs - (b+C) \| g \|_{L^2}^2 
\end{align*}
where $b > 0$. 
If $0 \leqs t \leqs T$ and $\mu \in \ro$ this gives, 
writing $u(t) = u(t,\cdot)$ for brevity,
\begin{align*}
\partial_t \left( e^{- 2 \mu t} \| u(t) \|_{L^2}^2 \right)
& = 2 e^{- 2 \mu t} \left(  \re \left( \partial_t u(t), u(t) \right)  - \mu \| u(t) \|_{L^2}^2 \right) \\
& = 2 e^{- 2 \mu t} \left(  \re \left( f(t), u(t) \right)  - \re \big( (i a + \mu)^w  u(t), u(t) \big) \right) \\
& \leqs 2 e^{- 2 \mu t} \re \left( f(t), u(t) \right)
\end{align*}
provided $\mu \geqs b + C$.

Integration gives for any $0 \leqs \nu \leqs t$
\begin{align*}
e^{- 2 \mu \nu} \| u(\nu) \|_{L^2}^2
& \leqs \| u(0) \|_{L^2}^2 + 2 \int_0^\nu e^{- 2 \mu \tau } \| f(\tau) \|_{L^2} \, \| u( \tau ) \|_{L^2} \, \dd \tau \\
& \leqs \| u(0) \|_{L^2}^2 + 2 \int_0^t e^{- 2 \mu \tau } \| f(\tau) \|_{L^2} \, \| u( \tau ) \|_{L^2} \, \dd \tau \\
& \leqs \| u(0) \|_{L^2}^2 + 2 M(t) \int_0^t e^{- \mu \tau} \| f(\tau) \|_{L^2}  \, \dd \tau
\end{align*}
with 
\begin{align*}
M(t) = \sup_{0 \leqs \tau \leqs t} e^{- \mu \tau} \| u(\tau) \|_{L^2}. 
\end{align*}

Thus 
\begin{align*}
\left( M(t) - \int_0^t e^{- \mu \tau} \| f(\tau) \|_{L^2}  \, \dd \tau \right)^2
\leqs \left(\| u(0) \|_{L^2} +  \int_0^t e^{- \mu \tau } \| f(\tau) \|_{L^2}  \, \dd \tau \right)^2
\end{align*}
which yields
\begin{align*}
e^{- \mu t} \| u(t) \|_{L^2}
& \leqs M(t) 
\leqs \| u(0) \|_{L^2} + 2 \int_0^t e^{- \mu \tau } \| f(\tau) \|_{L^2}  \, \dd \tau. 
\end{align*}
We have now shown \eqref{eq:wellposedinequality} for $s = 0$ and $c = \mu$. 

Next let $s \in \ro$ and 
$u \in  C([0,T], M_{s + 1 + \sigma} ) \cap C^1 ([0,T], M_s)$. 
Then $\partial_t u, a^w u \in C ([0,T], M_s)$, 
again appealing to Proposition \ref{prop:continuitySobolev}, and 
the conclusion \eqref{eq:fregular} follows. 
By the proof of Lemma \ref{lem:locopcalculus} we know that $A_s = a_1^w$ with $a_1 \in G^{s,\sigma}$. 
Proposition \ref{prop:continuitySobolev} yields
\begin{align*}
& A_s u \in C ([0,T], M_{\sigma+1}) \cap C^1([0,T], L^2), 
\quad a^w A_s u \in C ([0,T], L^2) \\
& \Longrightarrow \quad \partial_t A_s u + i a^w A_s u \in C ([0,T], L^2). 
\end{align*}
By Lemma \ref{lem:locopcalculus} the symbol $b \in G^{1 + \sigma, \sigma}$ defined by 
$b^w = A_s \, a^w A_s^{-1}$ satisfies $\im \, b \leqs C_2$ for some $C_2 > 0$. 
The inequality \eqref{eq:wellposedinequality} with $a=b$ and $s = 0$ thus gives
\begin{equation*}
\| A_s u(t) \|_{L^2} 
\lesssim e^{c t} \| A_s u(0) \|_{L^2} + \int_0^t e^{c (t-\tau)} \| \partial_t A_s u(\tau) + i b^w A_s u(\tau) \|_{L^2} \dd \tau
\end{equation*}
which finally yields
\begin{align*}
\| u(t) \|_{M_s} 
\asymp  \| A_s u(t) \|_{L^2}
& \lesssim e^{c t} \| A_s u(0) \|_{L^2} + \int_0^t e^{c (t-\tau)} \| A_s \left( \partial_t u(\tau) + i A_s^{-1} b ^w A_s u(\tau) \right) \|_{L^2} \dd \tau \\
& \asymp e^{c t} \| u(0) \|_{M_s} + \int_0^t e^{c (t-\tau)} \| \partial_t u(\tau) + i a^w u(\tau) \|_{M_s} \dd \tau. 
\end{align*}
\end{proof}

\begin{rem}\label{rem:timereversal}
If we strengthen the assumption \eqref{eq:imsymbolupperbound} with a lower bound as 
\begin{equation}\label{eq:imsymbolbound}
- C \leqs \im \, a (x,\xi) \leqs C, \quad (x,\xi) \in T^* \rr d, 
\end{equation}
for $C > 0$, then the time direction may be reversed in Lemma \ref{lem:wellposedness}. 
More precisely the lower bound in \eqref{eq:imsymbolbound}
yields the estimate
\begin{equation*}
\re ( i a^w g, g) \leqs (b+C) \| g \|_{L^2}^2, \quad g \in \cS.
\end{equation*}
Straightforward modifications of the argument in the proof for the case $s = 0$ leads to the estimate
\begin{align*}
\| u(-t) \|_{L^2}
\lesssim e^{c t} \| u(0) \|_{L^2} + \int_{-t}^0 e^{c (t+\tau) } \| f(\tau) \|_{L^2}  \, \dd \tau
\end{align*}
for $c > 0$ and $0 \leqs t \leqs T$. 
Taking into account Remark \ref{rem:abssymbolbound}
a statement replacing \eqref{eq:wellposedinequality} can then be formulated as follows.
If $s \in \ro$, $u \in  C([-T,T], M_{s + 1 + \sigma} ) \cap C^1 ([-T,T], M_s)$ then
\begin{equation*}
\| u(t) \|_{M_s} \lesssim e^{c |t|} \| u(0) \|_{M_s} + \int_{|\tau| \leqs |t| } e^{c (|t| + |\tau|)} \| f(\tau) \|_{M_s} \dd \tau
\end{equation*}
for $-T \leqs t \leqs T$, where $c > 0$. 
\end{rem}

The final tool for the proof of existence and uniqueness of a solution to \eqref{eq:anisoschrodeq} we need is the following 
approximation lemma. 

\begin{lem}\label{lem:approxmodulation}
Let $s \in \ro$. 
If $f \in L^1([0,T], M_s)$ then there exists a sequence $(f_n)_{n \geqs 1} \subseteq C_c^\infty ( (0,T), \cS( \rr d) )$
such that 
\begin{equation}\label{eq:regconvergence1}
\lim_{n \to + \infty} \|f_n - f \|_{L^1([0,T], M_s)} = 0. 
\end{equation}
\end{lem}

\begin{proof}
Since $C([0,T], M_s) \subseteq L^1([0,T], M_s)$ is dense 
we may assume $f \in C([0,T], M_s)$, and by 
Lemma \ref{lem:approximation} we may assume $f \in C([0,T], \cS)$. 
Thus we have 
\begin{equation}\label{eq:translationlimit}
\lim_{n \to +\infty} 
\sup_{ |\tau| \leqs \frac12}
\int_{\frac{\tau+1}{n}}^{T + \frac{\tau-1}{n}} \left\| f \left( t - \frac{\tau}{n}\right) - f(t) \right\|_{M_s} \, \dd t = 0. 
\end{equation}

We regularize $f$ with respect to $t \in [0,T]$ as 
\begin{equation*}
f_n (t) = \psi_n * \left( f \chi_n \right) (t) \in C_c^\infty( (0,T), \cS(\rr d))
\end{equation*}
where $\chi_n \in C_c^\infty(\ro)$ is the indicator function for the interval $[\frac1n,T-\frac1n] \subseteq \ro$, 
$\psi \in C_c^\infty(\ro)$, $\psi \geqs 0$, $\supp \psi \subseteq [-\frac12,\frac12]$, $\int_\ro \psi (x) \, \dd x = 1$, and
$\psi_n(x) = n \psi (n x)$. 

Writing 
\begin{align*}
f_n (t) - f (t)
= \int_{-\frac1{2n}}^\frac1{2n} \psi_n(\tau) \Big( \big( f (t - \tau) - f(t) \big) \chi_n (t - \tau) + f(t) \big( \chi_n (t - \tau) - 1 \big) \Big) \, \dd \tau
\end{align*}
we may estimate 
\begin{equation*}
\int_0^T \| f_n (t) - f (t) \|_{M_s} \, \dd t \leqs I_n + J_n
\end{equation*}
where 
\begin{align*}
I_n 
& = \int_{-\frac1{2n}}^\frac1{2n} \psi_n(\tau) \int_{0}^T \| f (t-\tau) - f(t) \|_{M_s} \chi_n (t-\tau) \, \dd t \, \dd \tau \\
& = \int_{-\frac1{2n}}^\frac1{2n} \psi_n(\tau) \int_{\tau+\frac1n}^{T+\tau-\frac1n} \| f (t-\tau) - f(t) \|_{M_s} \, \dd t \, \dd \tau \\
& = \int_{-\frac12}^\frac12 \psi(\tau) \int_{\frac{\tau+1}{n}}^{T + \frac{\tau-1}{n}} \left\| f \left( t- \frac{\tau}{n} \right) - f(t) \right\|_{M_s} \, \dd t \, \dd \tau \\
& \quad \longrightarrow 0, \quad n \to + \infty
\end{align*}
using \eqref{eq:translationlimit}.  
Finally 
\begin{align*}
J_n 
& = \int_{-\frac1{2n}}^\frac1{2n} \psi_n(\tau) \int_{0}^T \| f (t) \|_{M_s}  \left( 1 - \chi_n (t-\tau) \right) \, \dd t \, \dd \tau \\
& = \int_{-\frac12}^\frac12 \psi (\tau) \int_{0}^T \| f(t) \|_{M_s}  \left( 1 - \chi_n \left( t- \frac{\tau}{n} \right) \right) \, \dd t \, \dd \tau \\
& = \int_{-\frac12}^\frac12 \psi (\tau) \left( \int_{0}^{\frac{\tau+1}{n}} + \int_{T + \frac{\tau-1}{n}}^T  \right) \| f(t) \|_{M_s} \, \dd t \, \dd \tau \\
& \leqs \int_{-\frac12}^\frac12 \psi (\tau) \left( \int_{0}^{\frac{3}{2 n}} + \int_{T - \frac{3}{2 n}}^T  \right) \| f(t) \|_{M_s} \, \dd t \, \dd \tau \\
& = \left( \int_{0}^{\frac{3}{2 n}} + \int_{T - \frac{3}{2 n}}^T  \right) \| f(t) \|_{M_s} \, \dd t  \\
& \quad \longrightarrow 0, \quad n \to + \infty. 
\end{align*}
We have shown \eqref{eq:regconvergence1}. 
\end{proof}

\begin{rem}\label{rem:neginterval2}
Again we note (cf. Remark \ref{rem:neginterval1}) that 
Lemma \ref{lem:approxmodulation} is true also when we 
replace the interval $[0,T]$ with $[-T,T]$. 
\end{rem}

We have now finally arrived at a point where we may state and prove the existence and uniqueness of 
solutions to \eqref{eq:anisoschrodeq} that are continuous on the spaces $M_s$. 

\begin{thm}\label{thm:uniquesolutionMsigma}
Let $T > 0$, 
$k,m \in \no \setminus 0$, $\sigma = \frac{k}{m}$, $a \in G^{1 + \sigma,\sigma}$, suppose
\begin{equation*}
\im \, a (x,\xi) \leqs C, \quad (x,\xi) \in T^* \rr d, 
\end{equation*}
for $C > 0$, and let $s \in \ro$.  
If $u_0 \in M_s$ and $f \in L^1([0,T], M_s)$, 
then the equation \eqref{eq:anisoschrodeq} has a unique solution 
$u \in C ([0,T], M_s)$. 
\end{thm}

\begin{proof}
First we assume $u_0 \in \cS$ and $f \in C_c^\infty ( (0,T), \cS( \rr d) )$.  

Let $\psi \in C_c^\infty( (0,T) \times \rr d)$. 
With $\psi(t) = \psi(t,\cdot)$ we have 
$\psi (t), \partial_t \psi (t), \overline{a}^w \psi(t)  \in C( (0,T), \cS )$. 
Let $\nu \in \ro$.
Lemma \ref{lem:wellposedness} applied to
$t \mapsto \psi(T-t,\cdot)$ and $- \overline a$ gives 
\begin{align*}
\sup_{0 \leqs t \leqs T} \| \psi(t) \|_{M_{-\nu}}
& \lesssim \int_0^T \| - \partial_t \psi(T-t,\cdot) - i \overline a^w \psi(T-t,\cdot) \|_{M_{-\nu}} \, \dd t \\
& = \int_0^T \| \partial_t \psi(t,\cdot) + i \overline a^w \psi(t,\cdot) \|_{M_{-\nu}} \, \dd t. 
\end{align*}

This implies 
\begin{align*}
\left| \int_0^T ( f (t), \psi(t) ) \, \dd t \right|
& \leqs
\| f \|_{L^1([0,T], M_{\nu})} \| \psi \|_{L^\infty([0,T], M_{-\nu})} 
\lesssim \sup_{0 \leqs t \leqs T} \| \psi(t) \|_{M_{-\nu}} \\
& \lesssim \int_0^T \| \partial_t \psi(t,\cdot) + i \overline a^w \psi(t,\cdot) \|_{-\nu} \, \dd t. 
\end{align*}
Thus 
\begin{equation*}
L^1( (0,T], M_{-\nu}) \ni -\partial_t \psi(t,\cdot) - i \overline a^w \psi(t,\cdot)
\mapsto \int_0^T ( f (t), \psi(t) ) \, \dd t
\end{equation*}
is an anti-linear continuous functional. 
By the Hahn--Banach theorem it can be extended to a functional on $L^1( (0,T], M_{-\nu})$. 
From \cite[Theorem~IV.1 and Corollary~IV.4]{Diestel1} we know that the dual space of $L^1( (0,T], M_{-\nu})$ 
can be identified with $L^\infty( (0,T], M_\nu)$ through the natural pairing. 
Hence there exists $u \in L^\infty( (0,T], M_\nu) \subseteq \cD' ( (0,T) \times \rr d) $ such that 
\begin{align*}
\int_0^T ( f (t), \psi (t)) \, \dd t 
& = \int_0^T (u (t), -\partial_t \psi (t,\cdot) - i \overline a^w \psi (t,\cdot) ) \, \dd t \\
& = \int_0^T (\partial_t u (t) + i a^w u (t), \psi (t) ) \, \dd t. 
\end{align*}

It follows from this argument that $\partial_t u + i a^w u = f$
in $\cD'( (0,T) \times \rr d)$. 
From $u \in L^\infty( (0,T], M_\nu)$, $a \in G^{1 + \sigma,\sigma}$ and 
Proposition \ref{prop:continuitySobolev}
it follows that $\partial_t u \in L^\infty( (0,T], M_{\nu-(1+\sigma)})$. 
If we set $g(0) = u_0$ and
\begin{equation*}
g(t) = \int_0^t \partial_t u(\tau) \, \dd \tau + u_0, \quad 0 < t \leqs T,  
\end{equation*}
then $g \in C([0,T], M_{\nu-(1+\sigma)})$, and it follows from Lebesgue's differentiation theorem for Bochner integrals
\cite[Theorem~II.2.9]{Diestel1} that $g'(t) = \partial_t u(t)$ for almost all $t \in [0,T]$. 

If $\psi \in C_c^\infty( (0,T) \times \rr d)$ then we obtain from this
\begin{align*}
(  u, \partial_t \psi) 
& = - ( \partial_t u, \psi) 
= - \int_{\rr d} \int_0^T \partial_t g (t,x) \overline{\psi(t,x)} \, \dd t \ \dd x 
= (g, \partial_t \psi)
\end{align*}
which shows that $u = g \in C( [0,T], M_{\nu-(1+\sigma)})$. 
Now $\partial_t u + i a^w u = f$ and $a \in G^{1 + \sigma,\sigma}$ yields 
$u \in C^1( [0,T], M_{\nu-2(1+\sigma)})$. 
We may now apply Lemma \ref{lem:wellposedness} and conclude
\begin{equation*}
\sup_{0 \leqs t \leqs T} \| u(t) \|_{M_{\nu - 2(1+\sigma)}} 
\lesssim 
\| u_0 \|_{M_{\nu - 2(1+\sigma)}} 
+ \int_0^T \| f(t) \|_{M_{\nu - 2(1+\sigma)}} \, \dd t. 
\end{equation*}

Since $\nu \in \ro$ is arbitrary we get the following conclusion. 
If $u_0 \in \cS$ and $f \in C_c^\infty ( (0,T), \cS( \rr d) )$ then 
for any $\nu \in \ro$ there exists a solution $u \in C ( [0,T] , M_{\nu+1+\sigma}) \cap C^1 ( [0,T] , M_{\nu})$ 
to \eqref{eq:anisoschrodeq} such that 
\begin{equation}\label{eq:solutionSchwartz}
\sup_{0 \leqs t \leqs T} \| u(t) \|_{M_{\nu} }
\lesssim \| u_0 \|_{M_{\nu}} 
+ \int_0^T \| f(t) \|_{M_\nu} \, \dd t.
\end{equation}

If $u_0 \in M_s$ and $f \in L^1( [0,T], M_s)$ we take sequences $(u_n)_{n = 1}^{\infty} \subseteq \cS$ 
and $( f_n)_{n = 1}^{\infty} \subseteq C_c^\infty ((0,T), \cS( \rr d) )$
such that $\| u_n - u_0 \|_{M_s} \to 0$ 
and $\|f_n - f \|_{L^1([0,T], M_s)} \to 0$
as $n \to +\infty$. 
The former is possible due to \cite[Proposition~11.3.4]{Grochenig1}, and 
the latter thanks to Lemma \ref{lem:approxmodulation}. 
By the first part of the proof there exists a 
sequence $(u_n(t))_{n = 1}^{+\infty} \subseteq C( [0,T], M_{s+1+\sigma})$
such that $\partial_t u_n (t) + i a^w u_n (t) = f_n (t) $ and $u_n(0) = u_n$ for each $n \geqs 1$. 
By \eqref{eq:solutionSchwartz} with $\nu = s$ the sequence $(u_n (t) )_n$ is a Cauchy sequence in $C( [0,T], M_s)$. 
It follows that $( \partial_t u_n(t))_{n = 1}^{+\infty} \subseteq C( [0,T], M_{s})$
is a Cauchy sequence in $L^1( [0,T], M_{s-1-\sigma})$. 

The sequence $(u_n (t) )_n$ converges in $C( [0,T], M_s)$ to $u (t) \in C( [0,T], M_s)$,
and the sequence $(\partial_t u_n (t) )_n$ converges in $L^1( [0,T], M_{s-1-\sigma})$ to $v(t) \in L^1( [0,T], M_{s-1-\sigma})$, 
$v + i a^w u = f$ in $L^1( [0,T], M_{s-1-\sigma})$, and $u(0) = u_0$. 

If $\psi \in C_c^\infty( (0,T) \times \rr d)$ then 
\begin{align*}
(  \partial_t u,\psi) 
& = - ( u, \partial_t  \psi) 
= - \lim_{n \to \infty} \int_0^T ( u_n (t) , \partial_t  \psi (t,\cdot ) ) \, \dd t \\
& = \lim_{n \to \infty} \int_0^T ( \partial_t   u_n (t) , \psi (t,\cdot ) ) \, \dd t \\
& = \int_0^T ( v (t) , \psi (t,\cdot ) ) \, \dd t 
= (v, \psi)
\end{align*}
which shows that $v = \partial_t u$ in $L^1( [0,T], M_{s-1-\sigma})$. 
We conclude that $\partial_t u + i a^w u = f$ in $L^1( [0,T], M_{s-1-\sigma})$, 
$u(0) = u_0$, $u \in C( [0,T], M_{s})$, 
and
\begin{equation*}
\sup_{0 \leqs t \leqs T} \| u(t) \|_{M_{s} }
\lesssim \| u_0 \|_{M_{s}} 
+ \int_0^T \| f(t) \|_{M_s} \, \dd t. 
\end{equation*}

It remains to prove the uniqueness of the solution. 
Suppose $u \in C( [0,T], M_s)$,  
$\partial_t u + i a^w u = 0$ and $u(0) = 0$.
Then $u \in C^1( [0,T], M_{s-1-\sigma})$ by Proposition \ref{prop:continuitySobolev},  
and thus by Lemma \ref{lem:wellposedness}
we have $u(t) = 0$ in $M_{s-1-\sigma}$, 
which implies $u(t) = 0$ in $M_{s}$, 
for each $t \in [0,T]$. 
\end{proof}

By Remarks \ref{rem:neginterval1}, \ref{rem:timereversal}, \ref{rem:neginterval2}
and straightforward modifications in the proof of Theorem \ref{thm:uniquesolutionMsigma}
we may strengthen the assumption on $\im \,a$, reverse the time direction and obtain results for the equation
\begin{equation}\label{eq:anisoschrodeq2}
\left\{
\begin{array}{rl}
\partial_t u(t,x) + i a^w(x,D) u (t,x) & = f(t,x), \quad x \in \rr d, \quad t \in [-T,T] \setminus 0, \\
u(0,\cdot) & = u_0. 
\end{array}
\right.
\end{equation}

\begin{cor}\label{cor:uniquesolutionMsigma}
Let $T > 0$, 
$k,m \in \no \setminus 0$, $\sigma = \frac{k}{m}$, $a \in G^{1 + \sigma, \sigma}$, suppose
\begin{equation*}
|\im \, a (x,\xi) | \leqs C, \quad (x,\xi) \in T^* \rr d, 
\end{equation*}
for $C > 0$, and let $s \in \ro$.  
If $u_0 \in M_s$ and $f \in L^1([-T,T], M_s)$, 
the equation \eqref{eq:anisoschrodeq2} has a unique solution 
$u \in C ([-T,T], M_s)$. 
\end{cor}

Since
\begin{equation*}
L^1 ([-T,T], \cS(\rr {d})) = \bigcap_{s \in \ro} L^1 ([-T,T], M_s), \quad
C ([-T,T], \cS(\rr {d})) = \bigcap_{s \in \ro} C ([-T,T], M_s)
\end{equation*}
we get the following corollary taking into account \eqref{eq:SSSchwartz}.

\begin{cor}\label{cor:propagationSchwartz}
Let $T > 0$, 
$k,m \in \no \setminus 0$, $\sigma = \frac{k}{m}$, $a \in G^{1 + \sigma, \sigma}$, and suppose
\begin{equation*}
| \im \, a (x,\xi) | \leqs C, \quad (x,\xi) \in T^* \rr d, 
\end{equation*}
for $C > 0$. 
If $f \in L^1 ([-T,T], \cS(\rr {d}))$ and $u_0 \in \cS(\rr d)$ then the unique solution
to \eqref{eq:anisoschrodeq2} satisfies $u \in C ([-T,T], \cS(\rr {d}))$. 
\end{cor}

Finally we state a result dual to Corollary \ref{cor:propagationSchwartz}. 

\begin{cor}\label{cor:propagationtempered}
Let $T > 0$, 
$k,m \in \no \setminus 0$, $\sigma = \frac{k}{m}$, $a \in G^{1 + \sigma, \sigma}$, and suppose
\begin{equation*}
| \im \, a (x,\xi) | \leqs C, \quad (x,\xi) \in T^* \rr d, 
\end{equation*}
for $C > 0$. 
If $u_0 \in \cS'$ then by \eqref{eq:SSSchwartz} there exists $s \in \ro$ such that $u_0 \in M_s$. 
If $f \in L^1 ([-T,T], M_s)$ then the unique solution
to \eqref{eq:anisoschrodeq2} satisfies $u \in C ([-T,T], M_s)$. 
\end{cor}

\section{Propagation of anisotropic Gabor wave front sets for Schr\"odinger type equations}\label{sec:propagation}

The following lemma is an anisotropic version of \cite[Lemma~3.6]{Cappiello5}. 
Its proof is similar so we omit it (cf. also \cite[Lemma~3.2]{Rodino4}). 
The lemma will be used in the proof of Proposition \ref{prop:commutator1} which is essential 
for our main result Theorem \ref{thm:propagationWFs}. 

\begin{lem}\label{lem:pseudocalculusparameter}
Suppose $\sigma > 0$, $r_j(t) \in C([-T,T], G^{m_j,\sigma})$ for $j \geqs 0$, where $(m_j)_{j=0}^\infty \subseteq \ro$ is decreasing, 
$[-T,T] \ni t \mapsto \partial_t r_j(t)(z)$ is continuous for each $z \in T^* \rr d$,
and $\partial_t r_j(t) \in L^\infty([-T,T], G^{m_j,\sigma})$ for all $j \geqs 0$. 
Then there exists $r(t) \in C([-T,T], G^{m_0,\sigma})$ such that for any $n \geqs 1$
\begin{equation*}
r(t) - \sum_{j=0}^{n-1} r_j(t) \in C([-T,T], G^{m_n,\sigma}).
\end{equation*}
We write $r(t) \sim \sum_{j=0}^{\infty} r_j(t)$. 
\end{lem}

Note that $r(t)$ is unique modulo an element in $C([-T,T], \cS(\rr {2d}))$.

If $r_j(t) \in L^\infty ([-T,T], G^{m_j,\sigma})$ for $j \geqs 0$ we abuse the notation $r(t) \sim \sum_{j=0}^{\infty} r_j(t)$
to mean 
\begin{equation*}
r(t) - \sum_{j=0}^{n-1} r_j(t) \in L^\infty ([-T,T], G^{m_n,\sigma})
\end{equation*}
for $n \geqs 1$. 
In this interpretation $r(t)$ is unique modulo an element in $L^\infty([-T,T], \cS(\rr {2d}))$. 
Thus in Lemma \ref{lem:pseudocalculusparameter}
it holds $\partial_t r(t) \sim \sum_{j=0}^{\infty} \partial_t r_j(t)$ in the latter sense. 

In the next result we use the cutoff function $\psi_\delta$ introduced prior to Proposition \ref{prop:compsymbolflow}. 

\begin{prop}\label{prop:commutator1}
Let $\delta > 0$, 
$k,m \in \no \setminus 0$, $\sigma = \frac{k}{m}$, 
and suppose 
that $a \in G^{1 + \sigma,\sigma}$, 
$a \sim \sum_{j = 0}^{\infty} a_j$, 
where 
$a_0 \in C^\infty(\rr {2d} \setminus 0)$ is real-valued, 
\begin{equation}\label{eq:a0homogeneity1}
a_0( \lambda x, \lambda^\sigma \xi)  = \lambda^{1+\sigma } a_0(x,\xi), \quad \lambda > 0, \quad (x,\xi) \in T^* \rr d \setminus 0,  
\end{equation}
and $a_j \in G^{(1+\sigma) (1- 2 j ), \sigma}$ for $j \geqs 1$. 
The Hamiltonian flow $\chi_t: T^* \rr d \setminus 0 \to T^* \rr d \setminus 0$
corresponding to the Hamiltonian $a_0$ is then defined for $-T \leqs t \leqs T$ with $T > 0$.

If $q_0 \in G^{0,\sigma}$ 
then 
there exists a function $t \mapsto q(t)$ such that 
$q(0) = \psi_\delta q_0$, 
\begin{align}
& q(t) \in C([-T,T], G^{0,\sigma}), \label{eq:qcont} \\
& q(t) \sim \sum_{j = 0}^{\infty} q_j (t), \quad q_j (t) \in C([-T,T], G^{- 2 j  (1+\sigma),\sigma}), \label{eq:qasymp} \\
& \partial_t q(t) \sim \sum_{j = 0}^{\infty} \partial_t q_j (t), \quad \partial_t q_j (t) \in L^\infty([-T,T], G^{- 2 j (1+\sigma),\sigma}), \label{eq:qderasymp} \\
& q_0(t) (x,\xi) = \psi_\delta (x,\xi) q_0( \chi_{-t} (x,\xi)), \quad (x,\xi) \in T^* \rr d, \quad t \in [-T,T], \label{eq:q0}
\end{align}
and $r(t) \in L^\infty ([-T,T], \cS(\rr {2d}))$ where 
\begin{equation*}
r(t)^w = q(t)^w \left( \partial_t + i a^w \right)
- \left( \partial_t +i a^w \right) q(t)^w. 
\end{equation*}
\end{prop}

\begin{proof}
The claim that the Hamiltonian flow $\chi_t(x,\xi) \in T^* \rr d \setminus 0$
corresponding to the Hamiltonian $a_0$ is defined for $-T \leqs t \leqs T$ with the same parameter $T > 0$
for all initial data $(x,\xi) \in T^* \rr d \setminus 0$ is a consequence of 
Proposition \ref{prop:homogeneoushamiltonian}.

We will design $q(t)$ such that \eqref{eq:qcont}, \eqref{eq:qasymp}, \eqref{eq:qderasymp} and \eqref{eq:q0} are satisfied,
and, noting that $\partial_t q(t)^w = q(t)^w \partial_t + \left( \partial_t q(t) \right)^w$, 
\begin{equation}\label{eq:commutatorschwartz1}
r(t) = i \left( q(t) \wpr a - a \wpr q(t) \right) - \partial_t q(t) 
\in L^\infty ([-T,T], \cS(\rr {2d}) ).
\end{equation}

By \cite[Theorem~18.5.4]{Hormander1} we have
\begin{equation}\label{eq:commutatorasymp1}
\begin{aligned}
& i \left( q(t) \wpr a - a \wpr q(t) \right) (x,\xi) \\
& \sim \sum_{j=0}^{\infty} \frac{(-1)^{j+1}}{(2j + 1)! 2^{ 2 j}} \left( \la \partial_x, \partial_\eta \ra - \la \partial_y, \partial_\xi \ra  \right)^{2 j + 1} 
q(t)(x,\xi) a (y,\eta) 
\Big|_{(y,\eta) = (x,\xi)} \\
& \sim \{ q(t), a \}(x,\xi) + \sum_{j=1}^{\infty} \frac{(-1)^{j+1}}{(2j + 1)! 2^{ 2 j }} \left( \la \partial_x, \partial_\eta \ra - \la \partial_y, \partial_\xi \ra  \right)^{2 j + 1} 
q(t)(x,\xi) a (y,\eta) \Big|_{(y,\eta) = (x,\xi)}
\end{aligned}
\end{equation}
where we use the Poisson bracket notation
\begin{equation*}
\{ q (t),a \} = \la \nabla_\xi q(t), \nabla_x  a \ra - \la \nabla_x q(t), \nabla_\xi  a \ra = \la \J \nabla_{x,\xi} q(t), \nabla_{x,\xi} a \ra. 
\end{equation*}

If we introduce for $j \geqs 0$ the bilinear differential operator
\begin{equation}\label{eq:poissonorderj}
\{ f,g \}_{j} (x,\xi) 
= (-1)^j \left( \la \partial_x, \partial_\eta \ra - \la \partial_y, \partial_\xi \ra  \right)^{j} 
f(x,\xi) g (y,\eta) \Big|_{(y,\eta) = (x,\xi)}
\end{equation}
then $\{ f,g \}_{1} = \{ f,g\}$, so $\{ f,g \}_{j}$ extends the Poisson bracket to higher order differential operators. 
Note that for $j \geqs 0$
\begin{equation}\label{eq:symbolpoisson}
a \in G^{m,\sigma}, \quad b \in G^{n,\sigma} \quad \Longrightarrow \quad \{ a, b \}_j \in G^{m + n - j(1+\sigma),\sigma}. 
\end{equation}

The notation \eqref{eq:poissonorderj} allows us to abbreviate \eqref{eq:commutatorasymp1} as 
\begin{equation*}
i \left( q(t) \wpr a - a \wpr q(t) \right) 
\sim \sum_{j=0}^{\infty} \frac{(-1)^{j}}{(2j + 1)! 2^{ 2 j }} \, \{ q(t),a \}_{2j+1}. 
\end{equation*}
Inserting $a \sim \sum_{j = 0}^{\infty} a_j$ and $q(t) \sim \sum_{j = 0}^{\infty} q_j (t)$, 
and collecting terms of order $j \geqs 0$ gives
\begin{equation}\label{eq:commutatorasymp2}
i \left( q(t) \wpr a - a \wpr q(t) \right) 
\sim \sum_{j=0}^{\infty} \sum_{k + n + m = j} \frac{(-1)^{m}}{(2m + 1)! 2^{ 2 m }} \, \{ q_k (t),a_n \}_{2m+1}
\end{equation}
since $\{ q_k (t),a_n \}_{2m+1} \in C([-T,T], G^{- 2 j  (1+\sigma),\sigma})$ when $k + n + m = j$. 

The remainder \eqref{eq:commutatorschwartz1} can now be written $r(t) \sim \sum_{j=0}^\infty r_j (t)$
as an asymptotic sum in $L^\infty ([-T,T], G^{0,\sigma})$ with 
\begin{equation}\label{eq:remainderterm}
r_j (t) = \sum_{k + n + m = j} \frac{(-1)^{m}}{(2m + 1)! 2^{ 2 m }} \, \{ q_k (t),a_n \}_{2m+1} - \partial_t q_j(t) \in L^\infty ([-T,T], G^{-2j(1+\sigma),\sigma})
\end{equation}
for $j \geqs 0$. 
In the proof we show how to pick $\{ q_j(t)\}_{j=0}^\infty$ with the stated properties so that 
$r(t) \in L^\infty ([-T,T], \cS(\rr {2d}) )$.

Set for $(x,\xi) \in T^* \rr d$ and $t \in [-T,T]$
\begin{equation*}
q_0 (t)(x, \xi)  =  \psi_\delta (x,\xi) q_0 ( \chi_{-t} (x,\xi) )
\end{equation*}
so that \eqref{eq:q0} is satisfied. 
The purpose of the factor $\psi_\delta$ is to make $q_0 (t)$ a well defined smooth function also around $(0,0) \in T^* \rr d$
where $\chi_{-t}$ may not be smooth. 
For each $(x,\xi) \in T^* \rr d$, $t \mapsto q_0 (t)(x, \xi) \in C([-T,T])$. 
By Proposition \ref{prop:compsymbolflow} we have $q_0 (t) \in G^{0,\sigma}$ uniformly for all $t \in [-T,T]$.

We write $q_0 (t) ( \chi_t (x, \xi) ) =   \psi_\delta ( \chi_t (x, \xi) ) q_0 ( x,\xi )$. 
Then differentiation with respect to $t$,
$\partial_t \chi_t(x,\xi) = \J \nabla a_0 ( \chi_t(x,\xi) )$ 
(cf. \eqref{eq:hamiltoncompact})
and the Chain Rule
give for $(x,\xi) \in T^* \rr d \setminus 0$
\begin{equation}\label{eq:q0der0}
\begin{aligned}
 \{ a_0, \psi_\delta \} ( \chi_{t} (x,\xi) ) q_0 ( x,\xi )
& = \la \nabla  \psi_\delta ( \chi_{t} (x,\xi) ),  \J \nabla a_0 ( \chi_t(x,\xi) ) \ra q_0 ( x,\xi ) \\
& = \la \nabla  \psi_\delta ( \chi_{t} (x,\xi) ),  \partial_t \chi_t(x,\xi) \ra q_0 ( x,\xi ) \\
& = \left( \partial_t q_0(t) \right) ( \chi_t (x, \xi) ) + \la \nabla q_0 (t) ( \chi_{t} (x,\xi) ), \J \nabla a_0 ( \chi_{t} (x,\xi) ) \ra \\
& = \left( \partial_t q_0(t) \right) ( \chi_t (x, \xi) ) - \{ q_0 (t), a_0 \} ( \chi_{t} (x,\xi) ). 
\end{aligned}
\end{equation}
Thus for all $(x,\xi) \in T^* \rr d$
\begin{equation*}
\partial_t q_0 (t)(x, \xi)  =  \{ q_0(t), a_0 \}(x,\xi) - \{ \psi_\delta, a_0 \} (x,\xi) q_0 ( \chi_{-t} (x,\xi) ). 
\end{equation*}
Note that the right hand side is supported in $\rr {2d} \setminus \rB_{\frac{\delta}{2}}$, 
and the second term is compactly supported in $\overline \rB_\delta \subseteq T^* \rr d$. 

The function $[-T,T] \ni t \mapsto \partial_t q_0 (t)(x, \xi)$ is continuous for each $(x,\xi) \in T^* \rr d$. 
Indeed the continuity of the first term
\begin{align*}
[-T,T] \ni t \mapsto \{ q_0 (t), a_0 \} (x,\xi) = -  \la \nabla q_0 (t) (x,\xi), \J \nabla a_0 ( x,\xi ) \ra
\end{align*}
is a consequence of \eqref{eq:smoothflow} and the chain rule, 
and the continuity of the second term has been verified above.

From 
$q_0 (t) \in G^{0,\sigma}$ and \eqref{eq:symbolpoisson} it now follows that $\partial_t q_0(t) \in L^\infty([-T,T], G^{0,\sigma})$, and then
the continuity of $[-T,T] \ni t \mapsto \partial_t q_0 (t)(x, \xi)$
and
the mean value theorem gives $q_0 (t) \in C( [-T,T],G^{0,\sigma})$. 
By \eqref{eq:remainderterm} we have
\begin{equation*}
r_0(t) = \{ q_0 (t), a_0 \} - \partial_t q_0(t) = \{ \psi_\delta, a_0 \} q_0 \circ \chi_{-t} \in L^\infty( [-T,T], G^{0 , \sigma})
\end{equation*}
which implies that $\supp r_0(t) \subseteq \overline \rB_\delta \subseteq T^* \rr d$ for all $t \in [-T,T]$ 
so in fact we have $r_0(t) \in L^\infty( [-T,T], C_c^{\infty} )$. 
This means that the principal symbol of $r(t)$ vanishes: 
$r_0(t) \sim 0$. 

Next we eliminate the second highest order term in \eqref{eq:remainderterm} $r_1(t) \in L^\infty( [-T,T], G^{- 2(1+\sigma) , \sigma})$ 
by a proper choice of $q_1(t) \in C( [-T,T], G^{- 2(1+\sigma) , \sigma})$.
The term in $C( [-T,T], G^{- 2(1+\sigma) , \sigma})$ in \eqref{eq:commutatorasymp2} is 
\begin{equation*}
\{ q_0(t), a_1 \} + \{ q_1(t), a_0 \} - \frac1{24} \{ q_0 (t),a_0 \}_{3}.
\end{equation*}
Define 
\begin{equation}\label{eq:rho1def}
\rho_1(t) = \{ q_0(t), a_1 \}  
- \frac1{24} \{ q_0 (t),a_0 \}_{3}
\in C( [-T,T], G^{- 2(1+\sigma) , \sigma})
\end{equation}
so that $\rho_1(t) + \{ q_1(t), a_0 \} $ is the term in $C( [-T,T], G^{- 2(1+\sigma) , \sigma})$ in \eqref{eq:commutatorasymp2}. 
Define
\begin{equation}\label{eq:q1def1}
q_1(t) ( \chi_t (x,\xi) ) 
= \int_0^t \rho_1 (\tau) ( \chi_\tau (x,\xi) ) \, \dd \tau
\end{equation}
or equivalently
\begin{equation*}
q_1(t) ( x,\xi ) 
= \int_0^t \rho_1 (\tau) ( \chi_{\tau - t} (x,\xi) ) \, \dd \tau. 
\end{equation*}

From \eqref{eq:rho1def} and 
Proposition \ref{prop:compsymbolflow} it follows that $q_1 (t) \in G^{- 2(1+\sigma),\sigma}$ uniformly for all $t \in [-T,T]$.
We differentiate \eqref{eq:q1def1} with respect to $t$ which gives if $(x,\xi) \in T^* \rr d \setminus 0$
\begin{align*}
\rho_1 (t) ( \chi_t (x,\xi) )
& = (\partial_t q_1(t)) ( \chi_t (x,\xi) ) +  \la \nabla q_1 (t) ( \chi_t (x,\xi) ), \partial_t \chi_{t} (x,\xi)  \ra \\
& = (\partial_t q_1(t)) ( \chi_t (x,\xi) ) - \{ q_1 (t), a_0 \} ( \chi_{t} (x,\xi) ). 
\end{align*}
Thus $\partial_t q_1(t) - \{ q_1 (t), a_0 \} = \rho_1 (t)$
which implies $\partial_t  q_1 (t)\in L^\infty( [-T,T], G^{- 2(1+\sigma),\sigma})$. 

Let $(x,\xi) \in T^* \rr d$ be fixed. We know that $[-T,T] \ni t \mapsto \rho_1 (t) (x,\xi)$ is continuous, 
and the continuity of 
\begin{align*}
[-T,T] \ni t \mapsto \{ q_1 (t), a_0 \} (x,\xi) = -  \la \nabla q_1 (t) (x,\xi), \J \nabla a_0 ( x,\xi ) \ra
\end{align*}
is a consequence of the continuity of
\begin{equation*}
[-T,T] \ni t \mapsto \partial_{z} q_1(t) ( x,\xi ) 
= \int_0^t \partial_{z} \left( \rho_1 (\tau) ( \chi_{\tau - t} (x,\xi) )  \right) \, \dd \tau
\end{equation*}
for $z = x_j$ and $z = \xi_j$ for all $1 \leqs j \leqs d$.  
In turn, the latter is a consequence of 
\begin{align*}
& \partial_{z} \left( q_1(t + s) - q_1(t) \right)( x,\xi ) \\
& = \int_t^{t+s} \partial_{z} \left( \rho_1 (\tau) ( \chi_{\tau - t - s} (x,\xi) )  \right) \, \dd \tau
+ \int_0^t \partial_{z} \Big( \rho_1 (\tau) ( \chi_{\tau - t - s} (x,\xi) ) - \rho_1 (\tau) ( \chi_{\tau - t } (x,\xi) )  \Big) \, \dd \tau,
\end{align*}
the chain rule, and again \eqref{eq:smoothflow}.

It follows that $[-T,T] \ni t \mapsto \partial_t  q_1 (t) (x,\xi)$ is continuous for each $(x,\xi) \in T^* \rr d$. 
Combining this with $\partial_t  q_1 (t)\in L^\infty( [-T,T], G^{- 2(1+\sigma),\sigma})$ we may conclude that $q_1(t) \in C( [-T,T], G^{- 2(1+\sigma),\sigma})$. 
Referring to \eqref{eq:remainderterm} this implies that $r_1 (t) \in L^\infty ( [-T,T], G^{- 2(1+\sigma),\sigma})$ and 
\begin{align*}
r_1(t) & = \{ q_1(t), a_0 \} + \{ q_0 (t), a_1 \} 
- \frac1{24} \{ q_0 (t),a_0 \}_{3}
- \partial_t q_1(t) \\
& = \{ q_1(t), a_0 \}  + \rho_1 (t) - \partial_t q_1(t)
= 0 
\end{align*}
which shows that the choice of $q_1(t)$ in \eqref{eq:q1def1} indeed eliminates $r_1(t) \in L^\infty( [-T,T], G^{- 2(1+\sigma),\sigma})$. 

In a similar way we construct $q_j(t) \in C( [-T,T], G^{- 2 j (1+\sigma),\sigma} )$ for $j \geqs 2$ using $\{ q_k (t)  \}_{k=0}^{j-1}$,
by defining  
\begin{equation}\label{eq:rhojdef}
\rho_j(t) = 
\sum_{k + n + m = j, \ k < j} \frac{(-1)^{m}}{(2m + 1)! 2^{ 2 m }} \, \{ q_k(t),a_n \}_{2 m +1} 
\in C( [-T,T], G^{- 2 j (1+\sigma) , \sigma} )
\end{equation}
and 
\begin{equation}\label{eq:qjdef}
q_j(t) ( \chi_t (x,\xi) ) 
= \int_0^t \rho_j (\tau) ( \chi_\tau (x,\xi) ) \, \dd \tau. 
\end{equation}
As before $\partial_t q_j(t) \in L^\infty( [-T,T], G^{- 2 j (1+\sigma) , \sigma})$, $q_j(t) \in C( [-T,T], G^{- 2 j (1+\sigma) , \sigma})$, and 
$\partial_t q_j(t) - \{ q_j (t), a_0 \} = \rho_j (t)$,
which yields $r_j (t) \in L^\infty( [-T,T], G^{- 2j (1+\sigma), \sigma})$ (cf. \eqref{eq:remainderterm})
and
\begin{align*}
r_j(t) & = 
\sum_{k + n + m = j} \frac{(-1)^{m}}{(2m + 1)! 2^{ 2 m }} \, \{ q_k(t),a_n \}_{2 m +1} 
- \partial_t q_j(t) \\
& = \{ q_j(t), a_0 \} + \rho_j(t)  - \partial_t q_j(t) 
= 0.
\end{align*}

So $r_j(t) \sim 0$ for all $j \geqs 0$ which means that 
\begin{equation*}
r(t) \in \bigcap_{j \geqs 0} L^\infty( [-T,T], G^{- 2j (1+\sigma),\sigma})
= L^\infty( [-T,T], \cS(\rr {2d})). 
\end{equation*}

Finally defining $q(t)$ by \eqref{eq:qasymp}, 
Lemma \ref{lem:pseudocalculusparameter} shows that \eqref{eq:qcont} and \eqref{eq:qderasymp} hold. 
The claim $q(0) = \psi_\delta q_0$ is a consequence of $q_0(0) = \psi_\delta q_0$
and $q_j (0) \equiv 0$ for $j \geqs 1$.
\end{proof}

Combining Corollaries \ref{cor:uniquesolutionMsigma} and \ref{cor:propagationSchwartz} with Proposition \ref{prop:commutator1} we obtain 
our main result about propagation of anisotropic Gabor singularities for the evolution equation

\begin{equation}\label{eq:anisoschrodeq3}
\left\{
\begin{array}{rl}
\partial_t u(t,x) + i a^w(x,D) u (t,x) & = 0, \quad x \in \rr d, \quad t \in [-T,T] \setminus 0, \\
u(0,\cdot) & = u_0. 
\end{array}
\right.
\end{equation}

\begin{thm}\label{thm:propagationWFs}
Let 
$k,m \in \no \setminus 0$, $\sigma = \frac{k}{m}$, 
and suppose that $a \in G^{1 + \sigma,\sigma}$, 
$a \sim \sum_{j = 0}^{\infty} a_j$, 
where 
$a_0 \in C^\infty(\rr {2d} \setminus 0)$ is real-valued, 
\begin{equation}\label{eq:a0homogeneity2}
a_0( \lambda x, \lambda^\sigma \xi)  = \lambda^{1+\sigma } a_0(x,\xi), \quad \lambda > 0, \quad (x,\xi) \in T^* \rr d \setminus 0,  
\end{equation}
and $a_j \in G^{(1+\sigma) (1-2 j ), \sigma}$ for $j \geqs 1$. 
The Hamiltonian flow $\chi_t: T^* \rr d \setminus 0 \to T^* \rr d \setminus 0$
corresponding to the Hamiltonian $a_0$ is then defined for $-T \leqs t \leqs T$ with $T > 0$.
If $u_0 \in \cS'(\rr d)$ then \eqref{eq:anisoschrodeq3} has a unique solution denoted $\cK_t u_0$, and
we have  
\begin{equation*}
\WFgs ( \cK_t u_0) = \chi_t \WFgs (u_0), \quad t \in [-T,T]. 
\end{equation*}
\end{thm}

\begin{proof}
By Proposition \ref{prop:homogeneoushamiltonian} there exists $T > 0$ such that 
the Hamiltonian flow $\chi_t: T^* \rr d \setminus 0 \to T^* \rr d \setminus 0$
corresponding to the Hamiltonian $a_0$ is well defined for $-T \leqs t \leqs T$.

By \eqref{eq:SSSchwartz} there exists $s \in \ro$ such that $u_0 \in M_s$. 
From Corollary \ref{cor:uniquesolutionMsigma} we obtain the existence of a unique solution 
$u(t) = \cK_t u_0 \in C( [-T,T], M_s)$ 
to \eqref{eq:anisoschrodeq3}. 

Let $z_0 \in T^* \rr d \setminus \left( \WFgs (u_0) \cup \{ 0 \} \right)$. 
We may assume that $|z_0| = 1$. 
By \eqref{eq:WFcharacteristic} with $m = 0$ there exists $q_0 \in G^{0,\sigma}$ such that $q_0^w u_0 \in \cS$
and $z_0 \notin \charac_\sigma(q_0)$. 
By 
\eqref{eq:noncharacteristic} 
there exists a $\sigma$-conic neighborhood $\Gamma \subseteq T^* \rr d \setminus 0$ 
such that $z_0 \in \Gamma$, and $|q_0(x,\xi) | \geqs C > 0$ when $(x,\xi) \in \Gamma \setminus \rB_r$ for some $r > 0$. 

Let $0 < \delta \leqs |\chi_t (z_0)|$ for all $t \in [-T,T]$. 
By Proposition \ref{prop:commutator1} there exists $q(t) \in C([-T,T], G^{0,\sigma})$ such that 
$q(t) \sim \sum_{j = 0}^{\infty} q_j (t)$ with $q_j (t) \in C([-T,T], G^{- 2j(1+\sigma),\sigma})$ for $j \geqs 0$, 
$q_0(t) (x,\xi) = \psi_\delta (x,\xi) q_0( \chi_{-t} (x,\xi))$
and $r(t) \in L^\infty([-T,T], \cS(\rr {2d}))$ where 
\begin{equation*}
r(t)^w = q(t)^w \left( \partial_t + i a^w \right)
- \left( \partial_t +i a^w \right) q(t)^w. 
\end{equation*}
This gives 
\begin{equation*}
0 = q(t)^w \left( \partial_t + i a^w \right) u(t) 
= \left( \partial_t +i a^w \right) q(t)^w u(t) + r(t)^w u(t)
\end{equation*}
that is 
\begin{equation*}
\left( \partial_t +i a^w \right) q(t)^w u(t)  = - r(t)^w u(t). 
\end{equation*}
Set $f (t) = - r(t)^w u(t)$. 
By \eqref{eq:symbolintersection} and Proposition \ref{prop:continuitySobolev} we have for any $m \in \ro$
\begin{equation*}
\sup_{|t| \leqs T} \| f(t) \|_{M_{m + s}} 
\lesssim \sup_{|t| \leqs T} \| u(t) \|_{M_{s}} < \infty
\end{equation*}
which by \eqref{eq:SSSchwartz} implies that $f \in L^\infty ( [-T,T], \cS(\rr d)) \subseteq L^1 ( [-T,T], \cS(\rr d))$. 

Thus $q(t)^w u(t)$ solves the equation \eqref{eq:anisoschrodeq2}, 
and for the initial value we have
\begin{equation*}
q(0)^w u(0) 
= q_0(0)^w u(0) = \left( \psi_\delta q_0 \right)^w u_0 
= q_0^w u_0 + \left( (\psi_\delta-1) q_0 \right)^w u_0  \in \cS
\end{equation*}
due to $\psi_\delta-1 \in C_c^\infty(\rr {2d})$. 
At this point we may apply 
Corollary \ref{cor:propagationSchwartz} which gives 
$q(t)^w u(t) \in \cS(\rr d)$ for all $t \in [-T,T]$. 
We note that $q_0 (t) (\chi_t(x,\xi)) = q_0(x,\xi)$ if $|\chi_t(x,\xi)| \geqs \delta$, 
and $\chi_t \Gamma \subseteq T^* \rr d \setminus 0$ is a $\sigma$-conic neighborhood of $\chi_t(z_0)$, 
which is a consequence of Proposition \ref{prop:homogeneoushamiltonian}. 
This implies that $\chi_t(z_0) \notin \charac_\sigma ( q(t) )$, since the lower order terms $\{ q_j(t) \}_{j \geqs 1}$ in $q(t)$ decay on $T^* \rr d$. 
By 
\eqref{eq:WFcharacteristic} this means that $\chi_t(z_0) \notin \WFgs ( u(t) ) = \WFgs ( \cK_t u_0 ) $. 
We have shown 
\begin{equation*}
\WFgs ( \cK_t u_0) \subseteq \chi_t \WFgs (u_0), \quad t \in [-T,T]. 
\end{equation*}
The opposite inclusion follows from $\cK_t^{-1} =  \cK_{-t}$ and $\chi_t^{-1} =  \chi_{-t}$. 
\end{proof}

\begin{rem}\label{rem:comparisonorder}
In Theorem \ref{thm:propagationAiry1} the Hamiltonian has by Remark \ref{rem:comparison1} the form
\begin{equation*}
a = \sum_{j=0}^m a_j
\end{equation*}
where $a_j$ is real-valued for all $0 \leqs j \leqs m$, $\sigma = \frac1{m-1}$, 
$a_0 \in G^{1+\sigma,\sigma}$ satisfies \eqref{eq:a0homogeneity2}, and $a_j \in G^{\sigma(m-j),\sigma} = G^{(1+\sigma) \left( 1 - \frac{j}{m} \right),\sigma} \subseteq G^{1,\sigma}$ for $1 \leqs j \leqs m$.

In Theorem \ref{thm:propagationWFs} on the other hand $\sigma = \frac{k}{m}$ and
the Hamiltonian is 
$a \sim \sum_{j=0}^\infty a_j$, $a_0$ is again real-valued and satisfies \eqref{eq:a0homogeneity2}, 
and $a_j \in G^{(1+\sigma) (1-2 j ), \sigma} \subseteq G^{-(1+\sigma), \sigma}$ for $j \geqs 1$. 

Comparing Theorem \ref{thm:propagationAiry1} and Theorem \ref{thm:propagationWFs}
we may conclude that the former is not a particular case of the latter, 
due to the different assumptions on the perturbation $a - a_0$ of the Hamiltonian.   
\end{rem}

\section{Examples}\label{sec:examples}

Let again $\psi_\delta(x,\xi) = \fy(|x|^2 + |\xi|^2) \in C^\infty(\rr {2d})$ where $\fy \in C^{\infty}(\ro)$,
$0 \leqs \fy \leqs 1$, $\fy(t) = 0$ for $t \leqs \frac{\delta^2}{4}$ and $\fy(t) = 1$ for $t \geqs \delta^2$
for a given $\delta > 0$. 
Thus $\psi_\delta \big|_{\rB_{\frac{\delta}{2}}} \equiv 0$ and $\psi_\delta \big|_{\rr {2d} \setminus \rB_\delta} \equiv 1$. 

\begin{example}\label{ex:example1}
Let $\delta > 0$, $c \in \ro \setminus 0$, 
$k,m \in \no \setminus 0$, $\sigma = \frac{k}{m}$, and set 
\begin{equation*}
a(x,\xi) = c \psi_\delta (x,\xi) \left( |x|^{2k} + |\xi|^{2m} \right)^{\frac12 \left( \frac1k + \frac1m \right)}. 
\end{equation*}
Then
\begin{equation*}
a( \lambda x, \lambda^\sigma \xi) = \lambda^{1+\sigma} a(x,\xi), \quad \lambda \geqs 1, \quad (x,\xi) \in T^* \rr d, \quad | (x,\xi) | \geqs \delta, 
\end{equation*}
and $a \in G^{1+\sigma,\sigma}$. 
Theorem \ref{thm:propagationWFs} applies to this Hamiltonian. 
\end{example}

\begin{example}\label{ex:example2}
Let $c_1, c_2 \in \ro \setminus 0$,  $k \in \no \setminus 0$, and set
\begin{equation*}
a(x,\xi) = \psi_\delta (x,\xi) \left( c_1 |x|^{\frac{2k}{2k-1}} + c_2 |\xi|^{2 k} \right). 
\end{equation*}
With $\sigma = \frac1{2k-1}$ we have 
\begin{equation*}
a( \lambda x, \lambda^\sigma \xi) = \lambda^{1+\sigma} a(x,\xi), \quad \lambda \geqs 1, \quad (x,\xi) \in T^* \rr d, \quad | (x,\xi) | \geqs \delta. 
\end{equation*}

However we note that the singularity (non-smoothness) of the term $|x|^{\frac{2k}{2k-1}} = |x|^{1 + \sigma}$ at the origin is not annihilated by the cutoff function $\psi_\delta$ unless $k = 1$. 
For this purpose we would need a cutoff function that depends on $x$ only. 
But this type of cutoff function does not fit into the calculus with $G^{m,\sigma}$ symbols. 
So $a \notin G^{1+\sigma,\sigma}$ and we cannot apply Theorem \ref{thm:propagationWFs} to this Hamiltonian. 
\end{example}

\section*{Acknowledgments}
We thank the anonymous referee whose careful comments have spurred us to improve the manuscript. 
The first author is partially supported by the INDAM-GNAMPA project CUP E53C23001670001.
This work is partially supported by the MIUR project ``Dipartimenti di Eccellenza 2018-2022'' (CUP E11G18000350001).






\end{document}